\documentclass{amsart}
\usepackage{amsmath,amssymb,latexsym}
\usepackage{amscd}
\usepackage{ulem}
\usepackage{mathrsfs}
\usepackage{mathdots}
\usepackage{fancyhdr}

\usepackage[all]{xy}
\usepackage{amsthm}
\usepackage{amsfonts}

\newtheorem{theorem}{Theorem}[section]
\newtheorem{lemma}[theorem]{Lemma}
\newtheorem{proposition}[theorem]{Proposition}
\newtheorem{corollary}[theorem]{Corollary}

 \theoremstyle{definition}
 \newtheorem{definition}[theorem]{Definition}
 
 \newtheorem{remark}[theorem]{Remark}

 \numberwithin{equation}{section}
 \DeclareMathOperator{\End}{End}
 \DeclareMathOperator{\Hom}{Hom}
 \DeclareMathOperator{\Tor}{Tor}
 
 \DeclareMathOperator{\im}{Im}
 \DeclareMathOperator{\Ker}{Ker}
 \DeclareMathOperator{\Mod}{Mod}

\def\Ker{\mathop{\rm Ker}\nolimits}

\def\End{\mathop{\rm End}\nolimits}
\def\Ext{\mathop{\rm Ext}\nolimits}
\def\pd{\mathop{\rm pd}\nolimits}

\def\id{\mathop{\rm id}\nolimits}
\def\findim{\mathop{\rm fin.dim}\nolimits}

\def\mod{\mathop{\rm mod}\nolimits}
\def\Mod{\mathop{\rm Mod}\nolimits}
\def\ob{\mathop{\rm ob}\nolimits}
\def\Cen{\mathop{\rm Cen}\nolimits}
\def\Gfd{\mathop{\rm Gfd}\nolimits}
\def\Cone{\mathop{\rm Cone}\nolimits}

\def\D{\mathop{\rm D}\nolimits}
\def\K{\mathop{\rm K}\nolimits}
\def\H{\mathop{\rm H}\nolimits}

\def\op{\mathop{\rm op}\nolimits}
\def\Ab{\mathop{\rm Ab}\nolimits}

\def\H{\mathop{\rm H}\nolimits}

\begin{document}

\title{\footnotesize{Higher differential objects in additive categories}}

\author{Xi Tang}
\address{College of Science, Guilin University of Technology, Guilin 541004, Guangxi Province, P. R. China \\}
\email{tx5259@sina.com.cn}

\author{Zhaoyong Huang}
\address{Department of Mathematics, Nanjing University, Nanjing 210093, Jiangsu Province, P.R. China}
\email{huangzy@nju.edu.cn}
\urladdr{http://math.nju.edu.cn/~huangzy/}



\date{}

\begin{abstract}
Given an additive category $\mathcal{C}$ and an integer $n\geqslant 2$. We form a new additive category $\mathcal{C}[\epsilon]^n$
consisting of objects $X$ in $\mathcal{C}$ equipped with an endomorphism $\epsilon_X$ satisfying ${\epsilon^n_X}=0$.
First, using the descriptions of projective and injective objects in $\mathcal{C}[\epsilon]^n$, we not only
establish a connection between Gorenstein flat modules over a ring $R$ and $R[t]/(t^n)$, but also
prove that an Artinian algebra $R$ satisfies some homological conjectures if and only if so does $R[t]/(t^n)$.
Then we show that the corresponding homotopy category $\K(\mathcal{C}[\epsilon]^n)$ is a triangulated category
when $\mathcal{C}$ is an idempotent complete exact category. Moreover, under some conditions for an abelian category
$\mathcal{A}$, the natural quotient functor $Q$ from $\K(\mathcal{A}[\epsilon]^n)$ to the derived category
$\D(\mathcal{A}[\epsilon]^n)$ produces a recollement of triangulated categories. Finally, we prove that
if $\mathcal{A}$ is an Ab4-category with a compact projective generator,
then $\D(\mathcal{A}[\epsilon]^n)$ is a compactly generated triangulated category.
\end{abstract}

\maketitle

\section{\bf Introduction}
Let $R$ be an arbitrary associative ring with unit. A {\it differential $R$-module} is an $R$-module $M$
equipped with a square-zero endomorphism $\epsilon$, called the {\it differentiation} of $M$.  To be precise,
differential modules are exactly modules over {\it the ring of dual numbers} over $R$, that is, the ring
$R[\epsilon]:=R[t]/(t^2)$ (the factor ring of the polynomial ring $R[t]$ in one variable $t$ modulo the
ideal generated by $t^2$). Differential modules introduced in the monograph of Cartan and Eilenberg \cite{CE}
are ubiquitous in homological algebra, and were employed as a means to provide a convenient framework
for a unified treatment of some problems from ring theory and topology in work by Avramov, Buchweitz and Iyengar
\cite{ABI}. There are a lot of recent work on differential modules, see \cite{RZ,S1,S2,S3,W,XYY} and so on.
In particular, Xu, Yang and Yao \cite{XYY} introduced $n$-th differential modules with $n\geqslant 2$ such that
2-nd differential modules are exactly classical differential modules, and they proved that
an $n$-differential module is Gorenstein projective (resp. injective) if and only if its underlying
module is Gorenstein projective (resp. injective). It generalized a result about differential modules
by Wei \cite{W}.

Given an additive category $\mathcal{C}$. Stai \cite{S1} introduced a new additive category $\mathcal{C}[\epsilon]$
in the following way: An object in $\mathcal{C}[\epsilon]$ is a pair $(A, \epsilon_A)$ such that $A\in \mathcal{C}$
and $\epsilon_A\in \End_{\mathcal{C}}(A)$ has the property ${\epsilon_A}^2=0$; and a morphism
$f\in \Hom_{\mathcal{C}[\epsilon]}(A,B)$ is what one might expect, namely a morphism $f\in \Hom_{\mathcal{C}}(A,B)$
satisfying $\epsilon_Bf=f\epsilon_A$. It is easily seen that objects in
$\mathcal{C}[\epsilon]$ generalize differential modules. Inspired by the above facts, as a higher analogue of
$\mathcal{C}[\epsilon]$, we will introduce and study an additive category
$\mathcal{C}[\epsilon]^n$ with $n\geqslant 2$, such that objects in $\mathcal{C}[\epsilon]^n$ are a common generalization
of $n$-th differential modules and objects in $\mathcal{C}[\epsilon]$.
The paper is organized as follows.

In Section 2, we give some terminology and notations.

Let $\mathcal{C}$ be an additive category, and let $F:\mathcal{C}[\epsilon]^n\to \mathcal{C}$ be the forgetful functor
and $T:\mathcal{C}\to\mathcal{C}[\epsilon]^n$ the augmenting functor. In Section 3, we first prove that both pairs
$(F,T)$ and $(T,F)$ are adjoint pairs (Proposition \ref{prop: 3.1}). Let $(\mathcal{C},\mathscr{E})$ be an idempotent
complete exact category, and let $\mathscr{E}_F$ be the class of pairs of composable morphisms in $\mathcal{C}[\epsilon]^n$
that become short exact sequences in $\mathcal{C}$ via the forgetful functor $F$. Then, with the aid of Proposition \ref{prop: 3.1},
we have that $Y$ is projective (resp. injective) in $(\mathcal{C}[\epsilon]^n, \mathscr{E}_F)$ if and only if $Y\cong T(X)$
for some projective (resp. injective) object $X$ of $\mathcal{C}$ (Proposition \ref{prop: 3.6}). These two results are higher
analogue of \cite[Propositions 2.1 and 2.7]{S1} respectively. In the latter part of this section, we give two applications of
Proposition \ref{prop: 3.6}. One of them states that for a ring $R$, a left $R[t]/(t^n)$-module $M$ is Gorenstein flat
if and only if it is Gorenstein flat as a left $R$-module (Theorem ~\ref{thm: 3.10}).
The other states that an Artinian algebra $R$ satisfies any of the finitistic dimension conjecture, strong Nakayama conjecture,
generalized Nakayama conjecture, Auslander-Gorenstein conjecture, Nakayama conjecture and Gorenstein symmetry conjecture if and only if
$R[t]/(t^n)$ satisfies the same conjecture (Theorem ~\ref{thm: 3.13}).

Let $\mathcal{C}$ be an exact category with trivial exact structure $\mathscr{E}^t$, and let $\mathscr{E}_F^t$ be the induced
exact structure via the forgetful functor $F$ in $\mathcal{C}[\epsilon]^n$. In Section 4, we prove that if $(\mathcal{C},\mathscr{E}^t)$
is an idempotent complete exact category, then $(\mathcal{C}[\epsilon]^n,\mathscr{E}_F^t)$
is a Frobenius category and its stable category $\underline{\mathcal{C}[\epsilon]^n}$ is a triangulated category (Proposition \ref{prop: 4.2}),
which coincides with the homotopy category $\K(\mathcal{C}[\epsilon]^n)$ (Theorem ~\ref{thm: 4.7}).

In Section 5 we introduce the derived category $\D(\mathcal{A}[\epsilon]^n)$ of $n$-th differential objects for an abelian
category $\mathcal{A}$. With mild assumptions on $\mathcal{A}$, we show that both $(\K^p(\mathcal{A}[\epsilon]^n), \K^a(\mathcal{A}[\epsilon]^n))$
and $(\K^a(\mathcal{A}[\epsilon]^n), \K^i(\mathcal{A}[\epsilon]^n)$ are stable $t$-structures, where
\linebreak
$\K^p(\mathcal{A}[\epsilon]^n)$, $\K^i(\mathcal{A}[\epsilon]^n)$ and $\K^a(\mathcal{A}[\epsilon]^n)$
are the categories of $K$-projective, $K$-injective and acyclic objects respectively; moreover, these allow us to derive from
the quotient functor $Q: \K(\mathcal{A}[\epsilon]^n) \to \D(\mathcal{A}[\epsilon]^n)$ a recollement of triangulated categories
(Theorem ~\ref{thm: 5.11}). In particular, if $\mathcal{A}$ is an Ab4-category with a compact projective generator,
then $\D(\mathcal{A}[\epsilon]^n)$ is a compactly generated triangulated category (Theorem ~\ref{thm: 5.17}).

\section{\bf Preliminaries}
Let $\mathcal{C}$ be an additive category and $n\geqslant 2$ a fixed positive integer.
An {\it $n$-th differential object} of $\mathcal{C}$ is a pair $(X, \epsilon_X)$, where $X\in\ob\mathcal{C}$
and $\epsilon_X\in \End_{\mathcal{C}}(X)$ satisfying ${\epsilon_X}^n=0$. Differential objects \cite{S1}
are exactly 2-nd differential objects. We define the category $\mathcal{C}[\epsilon]^n$ as follows:
The objects of $\mathcal{C}[\epsilon]^n$ are $n$-th differential objects, and the set of morphisms from
$(X, \epsilon_X)$ to $(Y,\epsilon_Y)$ consists of morphisms $f: X\to Y$ of $\mathcal{C}$ satisfying the equality
$f\epsilon_X=\epsilon_Yf$. The morphisms in $\mathcal{C}[\epsilon]^n$ are composed in the same way as the morphisms
in $\mathcal{C}$. It is easy to see that $\mathcal{C}[\epsilon]^n$ is also an additive category.

Let $R$ be a ring, and let $\Mod R$ be the category of left $R$-modules
and $\mod R$ the category of finitely generated left $R$-modules. Then we have $(\Mod R)[\epsilon]^n\cong \Mod (R[t]/(t^n))$.
Indeed, to an object $(X,\epsilon_X)\in (\Mod R)[\epsilon]^n$, associate the left $R[t]/(t^n)$-module $X$ with
$$(r_0+r_1t+\cdots +r_{n-1}t^{n-1})x:=r_0x+r_1\epsilon_X(x)+\cdots +r_{n-1}{\epsilon_X}^{n-1}(x).$$
Conversely, given a left $R[t]/(t^n)$-module $X$, we associate it with an $n$-th differential object $(X,\epsilon_X)$ in
$(\Mod R)[\epsilon]^n$ where $\epsilon_X(x):=tx$.

The following definition is cited from \cite{B}, see also \cite{K} and \cite{Q}.

\begin{definition}\label{df: 2.1}
Let $\mathcal{C}$ be an additive category. A {\it kernel-cokernel pair} $(i,p)$ in $\mathcal{C}$ is a pair of composable
morphisms $A\buildrel {i} \over \rightarrow B\buildrel {p} \over \rightarrow C$ such that $i$ is a kernel of $p$ and $p$
is a cokernel of $i$. We shall call $i$ an {\it admissible monic} and $p$ an {\it admissible epic}.

An {\it exact category} ($\mathcal{C},\mathscr{E}$) is an additive category $\mathcal{C}$ with a class $\mathscr{E}$ of
kernel-cokernel pairs which is closed under isomorphisms and satisfies the following axioms:
\begin{enumerate}
\item[{[E0]}] For all objects $C\in \mathcal{C}$, the identity morphism $1_C$ is an admissible monic.
\item[{[E0$^{\op}$]}] For all objects $C\in \mathcal{C}$, the identity morphism $1_C$ is an admissible epic.
\item[{[E1]}] The class of admissible monics is closed under compositions.
\item[{[E1$^{\op}$]}] The class of admissible epics is closed under compositions.
\item[{[E2]}] The push-out of an admissible monic along an arbitrary morphism exists and yields
an admissible monic.
\item[{[E2$^{\op}$]}] The pull-back of an admissible epic along an arbitrary morphism exists and yields
an admissible epic.
\end{enumerate}
Elements of $\mathscr{E}$ are called {\it short exact sequences}.

Let ($\mathcal{C},\mathscr{E}$) and ($\mathcal{C}',\mathscr{E}'$) be exact categories. An (additive) functor
$F: \mathcal{C} \to \mathcal{C}'$ is called {\it exact} if $F(\mathscr{E})\subseteq\mathscr{E}'$.
Let $\Ab$ be the category of abelian groups with the canonical exact structure. An object $P$ of
an exact category $\mathcal{C}$ is called {\it projective} if the represented functor
$\Hom_{\mathcal{C}}(P,-):\mathcal{C} \to \Ab$ is exact. Dually an object $I$ of an exact category $\mathcal{C}$
is called {\it injective} if the corepresented functor $\Hom_{\mathcal{C}}(-, I):\mathcal{C} \to \Ab$ is exact.
An exact category ($\mathcal{C}$, $\mathscr{E}$) has {\it enough projectives} if for any $X\in \ob\mathcal{C}$
there exists an admissible epic $P\to X$ in $\mathcal{C}$ with $P$ projective. Dually, ($\mathcal{C}$, $\mathscr{E}$)
has {\it enough injectives} if for any $X\in \ob\mathcal{C}$ there exists an admissible monic $X\to I$ in $\mathcal{C}$
with $I$ injective. An exact category is {\it Frobenius} if it has enough projectives and injectives
and moreover the projectives coincide with the injectives. For a Frobenius category $\mathcal{C}$, the corresponding
{\it stable category $\underline{\mathcal{C}}$} is the category whose objects are the objects of $\mathcal{C}$
and morphisms are given by, for any $A,B\in \ob\mathcal{C}$, $\Hom_{\underline{\mathcal{C}}}(A,B)=\Hom_{\mathcal{C}}(A,B)/P(A,B)$,
where $P(A,B)$ is the subgroup of morphisms $A\to B$ that factor through a projective object of $\mathcal{C}$ (see \cite{Ha}).
For basic notions and terminology on triangulated or derived categories we refer to \cite{Ha} and \cite{We}.
\end{definition}

\section{\bf From $\mathcal{C}$ to $\mathcal{C}[\epsilon]^n$}

Let $\mathcal{C}$ be an additive category and $n\geqslant 2$.
We introduce two functors between $\mathcal{C}$ and $\mathcal{C}[\epsilon]^n$, which will be
used for a complete description of the projective and injective objects of $\mathcal{C}[\epsilon]^n$.
\begin{enumerate}
\item[(1)] The forgetful functor $F: \mathcal{C}[\epsilon]^n\to \mathcal{C}$ is defined on the objects
$(X, \epsilon_X)$ of $\mathcal{C}[\epsilon]^n$ by $F(X, \epsilon_X)=X$ and on the morphisms $f$
in $\mathcal{C}[\epsilon]^n$ by $F(f)=f$.
\item[(2)] We define the augmenting functor $T: \mathcal{C} \to \mathcal{C}[\epsilon]^n$, which takes an object
$X$ of $\mathcal{C}$ to the object $T(X)=(X^{\oplus n},\epsilon_{X^{\oplus n}})$ of
$\mathcal{C}[\epsilon]^n$ with $X^{\oplus n}=\underbrace {X\oplus X\oplus \cdots \oplus X}_n$ and
$$\epsilon_{X^{\oplus n}}:= \left(
\begin{array}{ccccc}
0& 0& 0&\cdots &0 \\
1& 0& 0&\cdots &0 \\
0& 1& 0&\cdots &0 \\
\vdots&\vdots &\ddots& \ddots&\vdots\\
0& 0& \cdots&1 &0\\
\end{array}
\right)_{n\times n,}$$
and takes a morphism $f$ in $\mathcal{C}$ to the morphism
$$\left(\begin{array}{ccccc}
f& 0&\cdots &0 \\
0& f&\cdots &0 \\
\vdots&\vdots &\ddots&\vdots\\
0& 0 &\cdots &f\\
\end{array}
\right)_{n\times n}$$
in $\mathcal{C}[\epsilon]^n$.
\end{enumerate}

\subsection{Two-sided adjoints and the consistency of properties}

In this subsection, we generalize the results about differential objects in \cite[Chapter 2]{S1}
to the case for higher differential objects.

\begin{proposition}\label{prop: 3.1}
Both pairs $(F,T)$ and $(T,F)$ are adjoint pairs.
\end{proposition}

\begin{proof}
Fix $(X, \epsilon_X)\in \ob\mathcal{C}[\epsilon]^n$ and $Y\in \ob\mathcal{C}$.

Firstly we show that $(F,T)$ is an adjoint pair. To this end, we need to find an isomorphism
$\phi:\Hom_{\mathcal{C}}(FX,Y)\to \Hom_{\mathcal{C}[\epsilon]^n}(X,TY)$ which is natural in both
$X$ and $Y$. Given $f\in \Hom_{\mathcal{C}}(FX,Y)$, we define
$$\phi(f)= \left(
\begin{array}{c}
f{\epsilon_X}^{n-1} \\
f{\epsilon_X}^{n-2} \\
\vdots \\
f\\
\end{array}
\right).$$
It is easy to verify that $\phi$ is well defined. Moreover, let
$\phi^{-1}:\Hom_{\mathcal{C}[\epsilon]^n}(X,TY)\to \Hom_{\mathcal{C}}(FX,Y)$ be given by
$$\phi^{-1} (g)=g_n\ \text{for} \
g=\left(
\begin{array}{c}
g_1 \\
g_2 \\
\vdots \\
g_n\\
\end{array}
\right)\in \Hom_{\mathcal{C}[\epsilon]^n}(X,TY).$$
It is obvious that $\phi^{-1}\phi=1$. On the other hand, since $g\in\Hom_{\mathcal{C}[\epsilon]^n}(X,TY)$,
the equality $\epsilon_{Y^{\oplus n}}g=g\epsilon_X$ implies $g_1\epsilon_X=0$ and $g_{i}\epsilon_X=g_{i-1}$
for any $2\leqslant i\leqslant n$. Thus we have
$$g_n{\epsilon_X}^i=g_{n-1}{\epsilon_X}^{i-1}=\cdots = g_{n-i+1}\epsilon_X=g_{n-i}$$
for any $1\leqslant i\leqslant n-1$, which implies $\phi\phi^{-1}=1$.
Now we will check the naturality of $\phi$, let $\alpha: (X, \epsilon_X)\to (X', \epsilon_{X'})$ be a morphism in
$\mathcal{C}[\epsilon]^n$. Then $\epsilon_{X'}\alpha=\alpha\epsilon_{X}$. For any morphism
$f\in \Hom_{\mathcal{C}}(FX',Y)$, we have
$$\Hom_{\mathcal{C}[\epsilon]^n}(\alpha,TY)\phi(f)
=\Hom_{\mathcal{C}[\epsilon]^n}(\alpha,TY)(\left(
\begin{array}{c}
f\epsilon_{X'}^{n-1} \\
f\epsilon_{X'}^{n-2} \\
\vdots \\
f\\
\end{array}\right))$$
$$=\left(\begin{array}{c}
f\epsilon_{X'}^{n-1}\alpha \\
f\epsilon_{X'}^{n-2}\alpha \\
\vdots \\
f\alpha\\
\end{array}\right)
=\left(
\begin{array}{c}
f\alpha\epsilon_{X}^{n-1} \\
f\alpha\epsilon_{X}^{n-2} \\
\vdots \\
f\alpha\\
\end{array}\right)
=\phi\Hom_{\mathcal{C}}(F\alpha,Y)(f).$$
On the other hand, let $\beta:Y\to Y'$ be a morphism in $\mathcal{C}$. For any morphism
$f\in \Hom_{\mathcal{C}}(FX,Y)$, we have
$$\Hom_{\mathcal{C}[\epsilon]^n}(X,T\beta)\phi(f)
=\Hom_{\mathcal{C}[\epsilon]^n}(X,T\beta)(\left(
\begin{array}{c}
f\epsilon_{X}^{n-1} \\
f\epsilon_{X}^{n-2} \\
\vdots \\
f\\
\end{array}\right))$$
$$=\left(
\begin{array}{c}
\beta f\epsilon_{X}^{n-1} \\
\beta f\epsilon_{X}^{n-2} \\
\vdots \\
\beta f\\
\end{array}\right)
=\phi \Hom_{\mathcal{C}}(FX,\beta)(f).$$
The arguments above induce the following commutative diagram
$$\xymatrix{ \Hom_{\mathcal{C}}(FX',Y) \ar[r]^{\Hom_{\mathcal{C}}(F\alpha,Y)} \ar[d]^{\phi}
& \Hom_{\mathcal{C}}(FX,Y) \ar[r]^{\Hom_{\mathcal{C}}(FX,\beta)} \ar[d]^{\phi}
& \Hom_{\mathcal{C}}(FX,Y')\ar[d]^{\phi} \\
\Hom_{\mathcal{C}[\epsilon]^n}(X',TY)\ar[r]^{\Hom_{\mathcal{C}[\epsilon]^n}(\alpha,TY)}
& \Hom_{\mathcal{C}[\epsilon]^n}(X,TY)\ar[r]^{\Hom_{\mathcal{C}[\epsilon]^n}(X,T\beta)}
& \Hom_{\mathcal{C}[\epsilon]^n}(X,TY').}$$

For any $f=(f_1,f_2,\cdots, f_n)\in \Hom_{\mathcal{C}[\epsilon]^n}(TY, X)$, let
$\psi: \Hom_{\mathcal{C}[\epsilon]^n}(TY, X)\to \Hom_{\mathcal{C}}(Y,FX)$ be given by
$\psi(f):=f_1\in \Hom_{\mathcal{C}}(Y,FX)$. Similarly, we have that $\psi$ is an isomorphism
which is natural in $X$ and $Y$. So $(T,F)$ is also an adjoint pair.
\end{proof}

These two functors $F$ and $T$ defined above are useful in transferring an exact structure
in the initial category $\mathcal{C}$ to $\mathcal{C}[\epsilon]^n$. Let $(\mathcal{C},\mathscr{E})$
be an exact category, and let $\mathscr{E}_F$ be the class of pairs of composable
morphisms in $\mathcal{C}[\epsilon]^n$ that become short exact sequences in $\mathcal{C}$
via the forgetful functor $F$.

\begin{lemma}\label{lem: 3.2}
Let $(\mathcal{C},\mathscr{E})$ be an exact category. Then the following statements hold.
\begin{enumerate}
\item $(\mathcal{C}[\epsilon]^n, \mathscr{E}_F)$ is an exact category.
\item $F: \mathcal{C}[\epsilon]^n\to \mathcal{C}$ is exact.
\item $T: \mathcal{C} \to \mathcal{C}[\epsilon]^n$ is exact.
\end{enumerate}
\end{lemma}

\begin{proof}
(1) Let us first show that $\mathscr{E}_F$ is a class of kernel-cokernel pairs.
Suppose that $A\buildrel {i} \over \rightarrow B\buildrel {p} \over \rightarrow C$ is a pair of
morphisms in $\mathcal{C}[\epsilon]^n$ such that $(i,p)\in \mathscr{E}_F$. Then $pi=0$ in
$\mathcal{C}[\epsilon]^n$. Let $i': A'\to B$ be a morphism in $\mathcal{C}[\epsilon]^n$ such that $pi'=0$.
Since $i$ is the kernel of $p$ in $\mathcal{C}$, there exists a unique morphism $\phi\in \Hom_{\mathcal{C}}(A',A)$
such that $i\phi=i'$. Thus
$$i\epsilon_A\phi=\epsilon_Bi\phi=\epsilon_Bi'=i'\epsilon_{A'}=i\phi\epsilon_{A'}.$$
As $i$ is the kernel of $p$ in $\mathcal{C}$, $i$ is left cancelable. So $\epsilon_A\phi=\phi\epsilon_{A'}$.
It means that $\phi$ is a morphism in $\mathcal{C}[\epsilon]^n$.
Consequently, $i$ is the kernel of $p$ in $\mathcal{C}[\epsilon]^n$.
By a dual argument, we have that $p$ is the cokernel of $i$ in $\mathcal{C}[\epsilon]^n$. Furthermore,
it is easy to observe that $\mathscr{E}_F$ is closed under isomorphisms. Now we turn to show
that $(\mathcal{C}[\epsilon]^n, \mathscr{E}_F)$ satisfies all the axioms of Definition ~\ref{df: 2.1}.

It is easy to verify directly that [E0] and [E0$^{\op}$] hold.

[E1] Let $i_1:A\to M$ and $i_2:M\to B$ be admissible monics in $\mathcal{C}[\epsilon]^n$.
Then they are also admissible monics in $\mathcal{C}$. Set $i:=i_2i_1$. Since $\mathcal{C}$ is an exact category,
we have a short exact sequence
$$A\buildrel {i} \over \rightarrow B\buildrel {p} \over \rightarrow C$$ in $\mathcal{C}$.
Since $p\epsilon_Bi=pi\epsilon_A=0$ and $p$ is a cokernel of $i$ in $\mathcal{C}$, there exists a morphism
$\epsilon_C: C\to C$ such that $\epsilon_C p=p\epsilon_B$. Thus
$${\epsilon_C}^n p={\epsilon_C}^{n-1}p\epsilon_B=\cdots =p{\epsilon_B}^n=0.$$
The fact that $p$ is right cancelable implies ${\epsilon_C}^n=0$. So $(C,\epsilon_C)$ is an $n$-th differential object
and $i$ is an admissible monic in $\mathcal{C}[\epsilon]^n$.

Dually, we get [E1$^{\op}$].

[E2] Given any $f\in \Hom_{\mathcal{C}[\epsilon]^n}(A,A')$ and an admissible monic $i\in \Hom_{\mathcal{C}[\epsilon]^n}(A,B)$.
There exists a push-out diagram
\begin{equation}\label{diag: (3.1)}
\begin{gathered}
$$\xymatrix{ A \ar[r]^{i} \ar[d]^{f} & B \ar[d]^{f'}  \\
A'\ar[r]^{i'} & B'}$$
\end{gathered}
\end{equation}
in $\mathcal{C}$ such that $i'$ is an admissible monic. Since
$$i'\epsilon_{A'}f=i'f\epsilon_{A}=f'i\epsilon_{A}=f'\epsilon_{B}i,$$
by the universal property of push-outs in $\mathcal{C}$ there exists a unique morphism $\epsilon_{B'}: B'\to B'$ in $\mathcal{C}$
such that $\epsilon_{B'}i'=i'\epsilon_{A'}$ and $f'\epsilon_{B}=\epsilon_{B'}f'$. Note that
$${\epsilon_{B'}}^{n}i'={\epsilon_{B'}}^{n-1}i'\epsilon_{A'}=\cdots =i'{\epsilon_{A'}}^{n}=0\ \text{and}$$
$${\epsilon_{B'}}^{n}f'={\epsilon_{B'}}^{n-1}f'\epsilon_{B}=\cdots =i'{\epsilon_{B}}^{n}=0.$$
By the universal property of push-outs in $\mathcal{C}$ again, ${\epsilon_{B'}}^{n}=0$ and thus
the diagram ~\eqref{diag: (3.1)} is a commutative diagram in $\mathcal{C}[\epsilon]^n$. Next, we shall prove that
the diagram ~\eqref{diag: (3.1)} enjoys the appropriate universal property also in $\mathcal{C}[\epsilon]^n$.
Given $(X, \epsilon_X)\in \ob\mathcal{C}[\epsilon]^n$ and two morphisms $u: A'\to X$, $v:B\to X$ of $\mathcal{C}[\epsilon]^n$
such that $uf=vi$. Then there exists a unique morphism $w: B'\to X$ in $\mathcal{C}$ such that $wi'=u$ and $wf'=v$. Then
$$(\epsilon_Xw-w\epsilon_{B'})i'=\epsilon_Xwi'-w\epsilon_{B'}i'=\epsilon_Xu-wi'\epsilon_{A'}=\epsilon_Xu-u\epsilon_{A'}=0\ \text{and}$$
$$(\epsilon_Xw-w\epsilon_{B'})f'=\epsilon_Xwf'-w\epsilon_{B'}f'=\epsilon_Xv-wf'\epsilon_{B}=\epsilon_Xv-v\epsilon_{B}=0.$$
It follows that $\epsilon_Xw-w\epsilon_{B'}=0$ by the universal property of push-outs, proving the existence of the push-out
of an admissible monic $i$ along $f$. The reasoning in [E1] will ensure that $i'$ is also an admissible monic.

Dually, we get [E2$^{\op}$].

(2) It follows directly from the definition of the exact structure in $\mathcal{C}[\epsilon]^n$.

(3) Let
$$A\buildrel {i} \over \rightarrow B\buildrel {p} \over \rightarrow C$$
be a short exact sequence in $\mathcal{C}$. Applying the functor $T$ to it yields a sequence
\begin{equation}\label{diag: (3.2)}
\begin{gathered}
$$\xymatrix{ A^{\oplus n} \ar[rrr]^{ \left(
\begin{array}{cccc}
i& 0& \cdots &0 \\
0& i& \cdots &0 \\
\vdots&\vdots &\ddots&\vdots\\
0& 0& \cdots&i\\
\end{array}
\right)}& & & B^{\oplus n} \ar[rrr]^{ \left(
\begin{array}{cccc}
p& 0& \cdots &0 \\
0& p& \cdots &0 \\
\vdots&\vdots &\ddots&\vdots\\
0& 0& \cdots&p\\
\end{array}
\right)}& & & C^{\oplus n}}$$
\end{gathered}
\end{equation}
in $\mathcal{C}[\epsilon]^n$. We deduce from \cite[Proposition 2.9]{B} that ~\eqref{diag: (3.2)} is also a short exact sequence
in $\mathcal{C}[\epsilon]^n$. Therefore $T$ is an exact functor.
\end{proof}

According to \cite{B,KS}, an additive category $\mathcal{C}$ is called {\it idempotent complete} if every idempotent endomorphism
$e=e^2$ of an object $X\in \ob\mathcal{C}$ splits, that is, there exists a factorization
$$X\buildrel {\pi} \over \longrightarrow Y\buildrel {\iota} \over \longrightarrow X$$
of $e$ with $\pi\iota=1_Y$; and $\mathcal{C}$ is called {\it weakly idempotent complete} if every retraction has a kernel or
equivalently every coretraction has a cokernel. In particular, any abelian category is idempotent complete.

\begin{lemma}\label{lem: 3.3}
The following statements hold.
\begin{enumerate}
\item If $\mathcal{C}$ is weakly idempotent complete, then $\mathcal{C}[\epsilon]^n$ is weakly idempotent complete.
\item If $\mathcal{C}$ is idempotent complete, then $\mathcal{C}[\epsilon]^n$ is idempotent complete.
\end{enumerate}
\end{lemma}

\begin{proof}
(1) Let $\mathcal{C}$ be weakly idempotent complete and $r: B\to C$ a retraction in
$\mathcal{C}[\epsilon]^n$. Indeed, $r$ is a retraction in $\mathcal{C}$, then it has a kernel
$i: A\to B$ in $\mathcal{C}$. Since $r\varepsilon_Bi=\varepsilon_Cri=0$, there exists a morphism
$\varepsilon_A:A\to A$ in $\mathcal{C}$ such that $\varepsilon_Bi=i\varepsilon_A$. Because
$$i{\varepsilon_A}^n=\varepsilon_Bi{\varepsilon_A}^{n-1}=\cdots ={\varepsilon_B}^ni=0,$$
we have ${\varepsilon_A}^n=0$ and $(A,\varepsilon_A)\in \ob\mathcal{C}[\epsilon]^n$.
Now let $i': A'\to B$ be a morphism in $\mathcal{C}[\epsilon]^n$ such that $ri'=0$.
Since $i$ is a kernel of $r$ in $\mathcal{C}$, there exists a unique morphism
$u: A'\to A$ in $\mathcal{C}$ such that $iu=i'$. Since
$$i(\varepsilon_Au-u\varepsilon_{A'})=i\varepsilon_Au-iu\varepsilon_{A'}=
\varepsilon_Biu-i'\varepsilon_{A'}=\varepsilon_{B}i'-\varepsilon_{B}i'=0,$$
we have that $\varepsilon_Au-u\varepsilon_{A'}=0$ and $u$ is a morphism in $\mathcal{C}[\epsilon]^n$.
So $i$ is also a kernel of $r$ in $\mathcal{C}[\epsilon]^n$.

(2) Assume that $\mathcal{C}$ is idempotent complete. Let $e$ be an idempotent endomorphism of $(X,\epsilon_X)$.
Restricting to $\mathcal{C}$, there exists an object $Y\in\ob\mathcal{C}$ and morphisms
$$X\buildrel {\pi} \over \longrightarrow Y\buildrel {\iota} \over \longrightarrow X$$
in $\mathcal{C}$ such that $\iota\pi=e$ and $\pi\iota=1_Y$. Take $\epsilon_Y=\pi\epsilon_X\iota$.
Using $\epsilon_Xe=e\epsilon_X$ and ${\epsilon_X}^n=0$, we have ${\epsilon_Y}^n=0$. It is straightforward to check that
$\pi$ and $\iota$ become morphisms in $\mathcal{C}[\epsilon]^n$.
\end{proof}

Next, we turn to study whether $\mathcal{C}[\epsilon]^n$ is closed under direct summands whenever
$\mathcal{C}$ enjoys the same property. This is established in the following proposition.

\begin{proposition}\label{prop: 3.4}
Let $\mathcal{C}$ be idempotent complete. Then for any $X\in \ob\mathcal{C}$,
the direct summands of $X$ are in 1-1 correspondence with the direct summands of $TX$ up to conjugation.
\end{proposition}

\begin{proof}
By Lemma ~\ref{lem: 3.3} and \cite[Definition 6.1]{B}, it is enough to prove that the idempotents of
$\End_{\mathcal{C}}(X)$ are in 1-1 correspondence (up to conjugation) with the idempotents of
$\End_{\mathcal{C}[\epsilon]^n}(TX)$. We define a map $\theta: \End_{\mathcal{C}}(X)\to
\End_{\mathcal{C}[\epsilon]^n}(TX)$ by
$$f\mapsto  \left(
\begin{array}{cccc}
f& 0& \cdots &0 \\
0& f& \cdots &0 \\
\vdots&\vdots &\ddots&\vdots\\
0& 0& \cdots&f\\
\end{array}
\right).$$
It is an injective map sending idempotents to idempotents. Now given $f=(a_{ij})\in \End_{\mathcal{C}[\epsilon]^n}(TX)$,
the requirement $f\epsilon_{TX}=\epsilon_{TX}f$ translates to
{\tiny $$\left(\begin{array}{ccccc}
a_{11}& a_{12}& a_{13}& \cdots &a_{1n} \\
a_{21}& a_{22}& a_{23}& \cdots &a_{2n} \\
a_{31}& a_{32}& a_{33}& \cdots &a_{3n} \\
\vdots&\vdots &\vdots &\ddots&\vdots\\
a_{n1}& a_{n2}&a_{n3}& \cdots &a_{nn}\\
\end{array}\right)
\left(\begin{array}{ccccc}
0& 0& 0&\cdots &0 \\
1& 0& 0&\cdots &0 \\
0& 1& 0&\cdots &0 \\
\vdots&\vdots &\ddots& \ddots&\vdots\\
0& 0& \cdots&1 &0\\
\end{array}\right)$$
$$=\left(\begin{array}{ccccc}
a_{12}& a_{13}&\cdots & a_{1n}&0 \\
a_{22}& a_{23}&\cdots & a_{2n}&0 \\
a_{32}& a_{33}&\cdots & a_{3n}&0 \\
\vdots&\vdots &\ddots &\ddots&\vdots\\
a_{n2}& a_{n3}&\cdots & a_{nn}&0 \\
\end{array}
\right)=\left(
\begin{array}{ccccc}
0& 0&0 &\cdots& 0 \\
a_{11}& a_{12}&a_{13}&\cdots & a_{1n} \\
a_{21}& a_{22}&a_{23}&\cdots & a_{2n} \\
\vdots&\vdots &\vdots &\ddots&\vdots\\
a_{{n-1}\,1}& a_{n-1\,2}&a_{n-1\,3}&\cdots & a_{n-1\,n} \\
\end{array}
\right)$$
$$=\left(\begin{array}{ccccc}
0& 0& 0&\cdots &0 \\
1& 0& 0&\cdots &0 \\
0& 1& 0&\cdots &0 \\
\vdots&\vdots &\ddots& \ddots&\vdots\\
0& 0& \cdots&1 &0\\
\end{array}\right)
\left(\begin{array}{ccccc}
a_{11}& a_{12}& a_{13}& \cdots &a_{1n} \\
a_{21}& a_{22}& a_{23}& \cdots &a_{2n} \\
a_{31}& a_{32}& a_{33}& \cdots &a_{3n} \\
\vdots&\vdots &\vdots &\ddots&\vdots\\
a_{n1}& a_{n2}&a_{n3}& \cdots &a_{nn}\\
\end{array}
\right).$$}
It follows that $a_{ij}=0$ for $i<j$ and
$$a_{11}=a_{22}=\cdots =a_{nn},\;\; a_{21}=a_{32}=\cdots = a_{n\,n-1},$$
$$a_{31}=a_{42}=\cdots = a_{n\,n-2},\;\; \cdots,\;\;  a_{n-1\,1}=a_{n2}.$$
Furthermore suppose that $f$ is an idempotent of $\End_{\mathcal{C}[\epsilon]^n}(TX)$.
Then the equality $f^2=f$ implies that $f=(a_{ij})$ may has the following form
$$f=\left(\begin{array}{ccccc}
e& 0& 0& \cdots &0 \\
a& e& 0& \cdots &0 \\
a_{31}& a& e& \cdots &0 \\
\vdots&\vdots &\vdots &\ddots&\vdots\\
a_{n1}& a_{n2}&a_{n3}& \cdots &e\\
\end{array}
\right)$$
with $e^2=e$ and $ae+ea=a$. Then $eae=0$. Let
$$g_1=\left(\begin{array}{ccccc}
1& 0& 0& \cdots &0 \\
ea-ae& 1& 0& \cdots &0 \\
0& ea-ae& 1& \cdots &0 \\
\vdots&\vdots &\vdots &\ddots&\vdots\\
0& 0&\cdots & ea-ae &1\\
\end{array}\right).$$
Then $g_1$ is obviously invertible with
$$g_1^{-1}=\left(\begin{array}{ccccc}
1& 0& 0& \cdots &0 \\
ae-ea& 1& 0& \cdots &0 \\
(ae-ea)^2& ae-ea& 1& \cdots &0 \\
\vdots&\vdots &\vdots &\ddots&\vdots\\
(ae-ea)^{n-1}& (ae-ea)^{n-2}&\cdots & ae-ea &1\\
\end{array}\right).$$
Hence we get
$$g_1fg_1^{-1}=\left(
\begin{array}{ccccc}
e& 0& 0& \cdots &0 \\
0& e& 0& \cdots &0 \\
a_{31}'& 0& e& \cdots &0 \\
\vdots&\vdots &\vdots &\ddots&\vdots\\
a_{n1}'& a_{n2}'&\cdots & 0 &e\\
\end{array}\right).$$Continuing in this way, we may find an automorphism $g$ of
$\End_{\mathcal{C}[\epsilon]^n}(TX)$ such that
$$gfg^{-1}=\left(
\begin{array}{ccccc}
e& 0& 0& \cdots &0 \\
0& e& 0& \cdots &0 \\
0& 0& e& \cdots &0 \\
\vdots&\vdots &\vdots &\ddots&\vdots\\
0& 0&\cdots & 0 &e\\
\end{array}\right).$$
\end{proof}

The following observation is useful in the sequel.

\begin{lemma} \label{lem: 3.5}
Let $(\mathcal{C},\mathscr{E})$ be an exact category and $(X, \epsilon_X)\in \ob\mathcal{C}[\epsilon]^n$.
Then there exists two short exact sequences in $\mathcal{C}[\epsilon]^n$ as follows:
\begin{equation}\label{diag: (3.3)}
\begin{gathered}
X'\buildrel {i'_X} \over \longrightarrow TX \buildrel {p'_X} \over \longrightarrow X,
\end{gathered}
\end{equation}
where $X'=X^{\oplus n-1}$,
$$i'_X=\left(\begin{array}{ccccc}
0& 0&  \cdots &-\epsilon_X \\
0& 0&  \cdots &1 \\
\vdots&\iddots  &\iddots&\vdots\\
-\epsilon_X& 1&  \cdots &0 \\
1& 0&  \cdots & 0 \\
\end{array}
\right)_{n\times (n-1)}, p'_X=(1,\epsilon_X,{\epsilon_{X}}^2,\cdots, {\epsilon_{X}}^{n-1}),$$
$$\varepsilon_{X'}=\left(\begin{array}{ccccc}
-\epsilon_X& 1& 0 & \cdots &0 \\
-{\epsilon_X}^2& 0& 1 & \cdots &0 \\
\vdots&\vdots &\vdots &\ddots &\vdots\\
-{\epsilon_X}^{n-2}& 0& 0 & \cdots &1 \\
-{\epsilon_X}^{n-1}& 0& 0 & \cdots &0 \\
\end{array}
\right)_{(n-1)\times (n-1)}, \,\, and$$
\begin{equation}\label{diag: (3.4)}
\begin{gathered}
X\buildrel {i''_X} \over \longrightarrow TX \buildrel {p''_X} \over \longrightarrow X'',
\end{gathered}
\end{equation}
where $X''=X^{\oplus n-1}$,
$$i''_X=\left(\begin{array}{c}
{\epsilon_X}^{n-1} \\
{\epsilon_X}^{n-2} \\
\vdots\\
{\epsilon_X} \\ \\
1 \\
\end{array}\right),
p''_X=\left(\begin{array}{ccccc}
0 & 0 & \cdots & 1 & -{\epsilon_X} \\
\vdots & \vdots & \iddots & \vdots & \vdots \\
0 & 1 & \cdots & 0 & 0 \\
1 & -{\epsilon_X} & \cdots & 0 & 0\\
\end{array}
\right)_{(n-1)\times n},
$$
$$\varepsilon_{X''}=\left(\begin{array}{ccccc}
0& 1& 0 & \cdots &0 \\
0& 0& 1 & \cdots &0 \\
\vdots&\vdots &\vdots &\ddots &\vdots\\
0& 0& 0 & \cdots &1 \\
-{\epsilon_X}^{n-1}& -{\epsilon_X}^{n-2}& -{\epsilon_X}^{n-3} & \cdots &-{\epsilon_X} \\
\end{array}
\right)_{(n-1)\times (n-1)}.$$
\end{lemma}

\begin{proof}
We just prove the existence of (3.3). Clearly $(\mathcal{C},\mathscr{E}_F)$ is an exact category
by Lemma ~\ref{lem: 3.2}(1). It is routine to check that $(X', \varepsilon_{X'})$ is an $n$-th differential
object and (3.3) is a sequence in $\mathcal{C}[\epsilon]^n$. We also have the following diagram
$$\xymatrix{
 X' \ar[rrrr]^{i'_X} \ar[d]^{1_{X'}}& & & & TX \ar[d]^h \ar[rrrr]^{p'_X} & & & & X \ar[d]^{1_X} \\
 X'   \ar[rrrr]^{i} & & & &TX  \ar[rrrr]^{p=(0,0,\cdots 0, 1)} & & & & X }$$ in $\mathcal{C}$ with
$$i=\left(\begin{array}{ccccc}
1& 0&  \cdots &0 \\
0& 1&  \cdots &0 \\
\vdots&\vdots  &\ddots&\vdots\\
0& 0&  \cdots &1 \\
0& 0&  \cdots & 0 \\
\end{array}
\right)_{n\times (n-1)},
h=\left(\begin{array}{ccccc}
0& 0& \cdots &0 &1 \\
0& 0& \cdots & 1 &\epsilon_X \\
\vdots&\vdots &\iddots &\iddots &\vdots\\
0& 1& \cdots & {\epsilon_X}^{n-3} &{\epsilon_X}^{n-2} \\
1& \epsilon_{X}& \cdots & {\epsilon_X}^{n-2} & {\epsilon_X}^{n-1} \\
\end{array}
\right)_{n\times n}.
$$ Since $h$ is an isomorphism and
$$X'\buildrel {i} \over \longrightarrow TX\buildrel {p} \over \longrightarrow X$$
is a short exact sequence in $\mathcal{C}$, we obtain that
$$X'\buildrel {i'_X} \over \longrightarrow TX\buildrel {p'_X} \over \longrightarrow X$$
is a short exact sequence in $\mathcal{C}[\epsilon]^n$.
\end{proof}

Now, we are able to describe completely the projective and injective objects of $\mathcal{C}[\epsilon]^n$.

\begin{proposition} \label{prop: 3.6}
Let $(\mathcal{C},\mathscr{E})$ be an idempotent complete exact category. Then we have
\begin{enumerate}
\item $P$ is a projective object of $(\mathcal{C}[\epsilon]^n, \mathscr{E}_F)$
if and only if $P\cong T(Q)$ for some projective object $Q$ of $\mathcal{C}$.
\item $I$ is an injective object of $(\mathcal{C}[\epsilon]^n, \mathscr{E}_F)$
if and only if $I\cong T(E)$ for some injective object $E$ of $\mathcal{C}$.
\end{enumerate}
\end{proposition}

\begin{proof}
It follows from Lemmas ~\ref{lem: 3.2} and ~\ref{lem: 3.3} that
$(\mathcal{C}[\epsilon]^n,\mathscr{E}_F)$ is an idempotent complete exact category.

(1) Let $Q$ be a projective object of $\mathcal{C}$ and $P\cong T(Q)$. Since
$$\Hom_{\mathcal{C}[\epsilon]^n}(P,-)\cong \Hom_{\mathcal{C}[\epsilon]^n}(T(Q),-)\cong \Hom_{\mathcal{C}}(Q,F(-))$$
by Proposition ~\ref{prop: 3.1} and since $F$ is an exact functor by Lemma ~\ref{lem: 3.2}(2),
$P$ is a projective object of $(\mathcal{C}[\epsilon]^n, \mathscr{E}_F)$.
Conversely, let $P$ be a projective object of $(\mathcal{C}[\epsilon]^n, \mathscr{E}_F)$. Since
$$\Hom_{\mathcal{C}}(F(P),-)\cong \Hom_{\mathcal{C}[\epsilon]^n}(P,T(-))$$
by Proposition ~\ref{prop: 3.1} and since $T$ is an exact functor by Lemma ~\ref{lem: 3.2}(3),
$F(P)$ is a projective object of $\mathcal{C}$. By Lemma ~\ref{lem: 3.5}, there exists a short exact sequence
$$P'\buildrel {i'_P} \over \longrightarrow TF(P) \buildrel {p'_P} \over \longrightarrow P$$
in $\mathcal{C}[\epsilon]^n$. It splits and $P$ is isomorphic to a direct summand of $TF(P)$.
It follows from Proposition ~\ref{prop: 3.4} that there exists a projective object $Q$ of $\mathcal{C}$
such that $P\cong T(Q)$.

(2) It is dual to (1).
\end{proof}

\subsection{Flat and Gorenstein flat modules}

We now use Proposition ~\ref{prop: 3.6} to prove the following corollary.

\begin{corollary}\label{cor: 3.7}
Let $R$ be a ring and $M\in \Mod(R[t]/(t^n))$. Then
$M$ is flat in $\Mod(R[t]/(t^n))$ if and only if $M\cong T(N)$ for some flat module $N$ in $\Mod R$.
\end{corollary}

\begin{proof}
If $M$ is flat in $\Mod(R[t]/(t^n))$, then $M\cong {\lim\limits_{\longrightarrow }}P_i$
with $\{P_i\}$ a family of projective modules in $\Mod(R[t]/(t^n))$. By Proposition ~\ref{prop: 3.6}(1),
there exists a projective left $R$-module $Q_i$ such that $P_i=T(Q_i)$ for any $i$.
Since the functor $T$ preserves direct limits, we have
$$M\cong {\lim\limits_{\longrightarrow }}P_i\cong {\lim\limits_{\longrightarrow }}T(Q_i)
\cong T({\lim\limits_{\longrightarrow }}Q_i).$$
As ${\lim\limits_{\longrightarrow }}Q_i$ is flat in $\Mod R$, the sufficiency follows.

Conversely, if $M\cong T(N)$ for some flat object $N$ in $\Mod R$, then $N\cong {\lim\limits_{\longrightarrow }}Q_i$
with $\{Q_i\}$ a family of projective modules in $\Mod R$. Thus
$$M\cong T(N)\cong T({\lim\limits_{\longrightarrow }}Q_i)\cong {\lim\limits_{\longrightarrow }}T(Q_i).$$
By Proposition ~\ref{prop: 3.6}(1), $T(Q_i)$ is projective in $\Mod(R[t]/(t^n))$. So $M$ is flat in $\Mod(R[t]/(t^n))$.
\end{proof}

For any $m\geqslant 0$, recall that a left and right Noetherian ring $R$ is called
{\it $m$-Gorenstein} if the left and right self-injective dimensions of $R$ are at most $m$,
and $R$ is called {\it Gorenstein} if it is $m$-Gorenstein for some $m$.
For a ring $R$, we use $\Cen(R)$ to denote the center of $R$.
Recall that a ring $R$ is called an {\it Artin algebra} if it is a finitely generated $\Cen(R)$-module
with $\Cen(R)$ a commutative Artin ring. Clearly a ring $R$ is an Artin algebra if and only if
it is a finitely generated $C$-module for some commutative Artin ring $C$ (\cite{AR1}).

\begin{corollary}\label{cor: 3.8}
For any ring $R$, we have
\begin{enumerate}
\item $R$ is left (resp. right) Noetherian if and only if $R[t]/(t^n)$ is left (resp. right) Noetherian.
\item For any $m\geqslant 0$, $R$ is $m$-Gorenstein if and only if $R[t]/(t^n)$ is $m$-Gorenstein.
\item $R$ is left (resp. right) perfect if and only if $R[t]/(t^n)$ is left (resp. right) perfect.
\item $R$ is left (resp. right) Artinian if and only if $R[t]/(t^n)$ is left (resp. right) Artinian.
\item $R$ is an Artin algebra if and only if $R[t]/(t^n)$ is an Artin algebra.
\item $R$ is left (resp. right) coherent if and only if $R[t]/(t^n)$ is left (resp. right) coherent.
\end{enumerate}
\end{corollary}

\begin{proof}
(1) By \cite[Theorem 1.1]{Ba2}, it suffices to show that any direct sum of injective modules
in $\Mod R$ is injective if and only if any direct sum of injective modules in $\Mod (R[t]/(t^n))$ is injective.

Let $R$ be left Noetherian and $\{I_i\}_{i\in I}$ a family of injective modules in $\Mod (R[t]/(t^n))$.
By Proposition ~\ref{prop: 3.6}(2), we have $I_i\cong T(E_i)$ for some injective module $E_i$ in $\Mod R$
for any $i\in I$. By \cite[Theorem 1.1]{Ba2}, $\oplus_{i\in I}E_i$ is injective in $\Mod R$.
Note that the functor $T$ preserves direct sums by Proposition ~\ref{prop: 3.1}. So
$$\oplus_{i\in I}I_i\cong \oplus_{i\in I}T(E_i)\cong T(\oplus_{i\in I}E_i)$$
is injective by Proposition ~\ref{prop: 3.6}(2) again.

Conversely, let $R[t]/(t^n)$ be left Noetherian and $\{E_i\}_{i\in I}$ a family of injective modules in $\Mod R$.
Then $T(\oplus_{i\in I}E_i)\cong \oplus_{i\in I}T(E_i)$ is injective in $\Mod (R[t]/(t^n))$.
By Proposition ~\ref{prop: 3.1} we have
$$\Hom_{R}(-, FT(\oplus_{i\in I}E_i))\cong \Hom_{R[\epsilon]^n}(T(-),T(\oplus_{i\in I}E_i)),$$
So $FT(\oplus_{i\in I}E_i)$ is injective in $\Mod R$. Furthermore, since $F$ is the forgetful functor,
$\oplus_{i\in I}E_i$ is injective in $\Mod R$.

(2) By (1), we have that $R$ is left and right Noetherian if and only if $R[t]/(t^n)$
is left and right Noetherian. Now using \cite[Theorem 12.3.1]{EJ} and \cite[Theorem 3.11(iii)]{XYY},
we get that $R$ is $m$-Gorenstein if and only if $R[t]/(t^n)$ is $m$-Gorenstein.

(3) We know from \cite[Theorem 28.4]{AF} that $R$ is left perfect if and only if every flat left $R$-module is projective.
Assume that $R[t]/(t^n)$ is left perfect and $M\in\Mod R$ is flat. Then $T(M)$ is a flat module in $\Mod (R[t]/(t^n))$ by
Corollary ~\ref{cor: 3.7}. So $T(M)$ is projective and there exists a projective left $R$-module $Q$ such that $T(M)\cong T(Q)$.
Thus $FT(M)$ is projective in $\Mod R$, and therefore $M$ is a projective left $R$-module. The converse may be proved similarly.

(4) Note that a ring $R$ is left (resp. right) Artinian if and only if it is left (resp. right) Noetherian and
right (resp. left) perfect (c.f. \cite[Theorem P]{Ba1} and \cite[Theorem 6]{Bj}). Thus the assertion follows from (1) and (3).

(5) It is easy to verify that $\Cen(R[t]/(t^n))=\Cen(R)[t]/(t^n)$. Thus by (4), we have that $\Cen(R)$ is a commutative Artinian ring
if and only if $\Cen(R[t]/(t^n))$ is a commutative Artinian ring. In addition, we have that $R$ is a finitely generated $\Cen(R)$-module
if and only if $R[t]/(t^n)$ is a finitely generated $\Cen(R[t]/(t^n))$-module. The assertion follows.

(6) By \cite[Theorem 2.1]{C}, $R$ is right coherent if and only if the direct product of any family of flat left $R$-modules is flat.
Assume that $R[t]/(t^n)$ is right coherent and $\{M_i\}_{i\in I}$ is a family of flat left $R$-modules. Since the functor $T$ preserves
direct products, $T(\prod_{i\in I} M_i)\cong \prod_{i\in I} T(M_i)$ is flat. By Corollary ~\ref{cor: 3.7}, there exists a flat left
$R$-modules $S$ such that $T(\prod_{i\in I} M_i)\cong T(S)$. Thus $\prod_{i\in I} M_i$ is also flat as a left $R$-module.
The converse may be proved similarly.
\end{proof}

\begin{remark}\label{rem: 3.9}
When we take $R$ in Corollary ~\ref{cor: 3.8}(2) to be $KQ$, that is, the path algebra over a field $K$,
then $KQ$ is 1-Gorenstein, and so $KQ[t]/(t^2)$ is also 1-Gorenstein. It recovers part of \cite[Theorem 2]{RZ}.
\end{remark}

We recall from \cite{EJ, EJT} that a left $R$-module $M$ is called {\it Gorenstein flat} if there exists an exact sequence
$${\bf F}:\cdots \to F_1\to F_0\to F^0\to F^1\to \cdots$$
in $\Mod R$ with all $F_i,F^i$ flat such that $M\cong \im (F_0\to F^0)$ and $E\otimes_{R}{\bf F}$ is exact for any injective
right $R$-module $E$. Furthermore, the {\it Gorenstein flat dimension} $\Gfd_RM$ of a left $R$-module $M$ is defined to be
$\inf\{n\geqslant 0\mid$ there exists an exact sequence
$$0 \to G_{n}\to G_{n-1}\to \cdots \to G_0\to M\to 0$$
in $\Mod R$ with all $G_i$ Gorenstein flat$\}$. If no such an integer exists, then set $\Gfd_RM=\infty$.
We write $(-)^+:=\Hom_{\mathbb{Z}}(-,\mathbb{Q}/\mathbb{Z})$, where $\mathbb{Z}$ is the additive
group of integers and $\mathbb{Q}$ is the additive group of rational numbers.

\begin{theorem}\label{thm: 3.10}
Let $R$ be a ring and $M\in \Mod(R[t]/(t^n))$. Then $M$ is Gorenstein flat in $\Mod(R[t]/(t^n))$ if and only if
$M$ is Gorenstein flat in $\Mod R$.
\end{theorem}

\begin{proof} Assume that $M$ is Gorenstein flat in $\Mod(R[t]/(t^n))$. Then there exists an exact sequence
\begin{equation}\label{diag: (3.5)}
\begin{gathered}
{\bf F}:\cdots \to F_1\to F_0\to F^0\to F^1\to \cdots
\end{gathered}
\end{equation}
in $\Mod(R[t]/(t^n))$ with all $F_i,F^i$ flat such that $M\cong \im (F_0\to F^0)$ and $I\otimes_{R[t]/(t^n)}{\bf F}$
is exact for any injective right $R[t]/(t^n)$-module $I$. Indeed, (3.5) is an exact sequence of flat left $R$-modules
by Corollary ~\ref{cor: 3.7}. Let $E$ be an injective right $R$-module. Then $T(E)$ is injective right $R[t]/(t^n)$-module
by Proposition ~\ref{prop: 3.6}(2). So $T(E)\otimes_{R[t]/(t^n)}{\bf F}$ and $(T(E)\otimes_{R[t]/(t^n)}{\bf F})^+$
are exact. By the adjoint isomorphism theorem, $\Hom_{R[t]/(t^n)}({\bf F}, T(E)^+)$, and hence $\Hom_{R[t]/(t^n)}({\bf F}, T(E^+))$,
is exact. It follows from \cite[Proposition 3.3]{XYY} that $\Hom_R({\bf F}, E^+)$ is exact.
By the adjoint isomorphism theorem again, $(E\otimes_R{\bf F})^+$ and $E\otimes_R{\bf F}$ are exact.
Consequently, we conclude that $M$ is Gorenstein flat as a left $R$-module.

Conversely, assume that $M$ is Gorensten flat in $\Mod R$ and $E$ is any injective right $R$-module.
Then there exists an exact sequence
$${\bf F}:0 \to M\buildrel {f^0} \over \longrightarrow F^0\buildrel {f^1} \over \longrightarrow F^1\to \cdots $$
in $\Mod R$ with all $F^i$ is flat such that $\Tor_{\geqslant 1}^R(E,M)=0$ and $E\otimes_R{\bf F}$ is exact.
Since all modules have flat covers by \cite[Theorem 3]{BBE},
there exists an exact sequence
$${\bf F'}: \cdots \to F'_1\to F'_0\to M\to 0$$ in $\Mod R$ with all $F'_i$ flat such that
$\Hom_R(Q, {\bf F'})$ is exact for any flat left $R$-module $Q$.
Notice that $E\otimes_R{\bf F'}$, and hence $(E\otimes_R{\bf F'})^+$, is exact, so
$\Hom_R({\bf F'},E^+)$ is also exact by the adjoint isomorphism theorem. Then we deduce from \cite[Lemma 3.7(ii)]{XYY}
that there exists an exact sequence
$${\bf S}: \cdots \to S_1\to S_0\to M\to 0$$
in $\Mod(R[t]/(t^n))$ with all $S_i$ flat such that $\Hom_{R[t]/(t^n)}({\bf S}, T(E)^+)$ is exact. By Proposition \ref{prop: 3.6}(2),
$\Hom_{R[t]/(t^n)}({\bf S},I^+)$ is exact for any injective right $R[t]/(t^n)$-module $I$. By the adjoint isomorphism theorem,
we have that $(I\otimes_{R[t]/(t^n)}{\bf S})^+$, and hence $I\otimes_{R[t]/(t^n)}{\bf S}$, is also exact.
It yields $\Tor_{\geqslant 1}^{R[t]/(t^n)}(I,M)=0$ for any injective right $R[t]/(t^n)$-module $I$.

On the other hand, there exists an exact sequence
\begin{equation}\label{diag: (3.6)}
\begin{gathered}
$$\xymatrix{0 \ar[r] & M \ar[rr]^{ \left(
\begin{array}{c}
f_0{\epsilon_{M}}^{n-1} \\
\vdots \\
f_0{\epsilon_{M}} \\
f_0\\
\end{array}
\right)}&  & TF_0 \ar[r]  & X \ar[r] & 0}$$
\end{gathered}
\end{equation}
in $\Mod(R[t]/(t^n))$. Since $E\otimes_R{\bf F}$ is exact,
any morphism in $\Mod R$ from $M$ to $E^+$ can be extended to $F_0$. Also it is easy to verify that
any morphism in $\Mod R$ from $M$ to $E^+$ can be extended to $TF_0$. Hence we have $\Tor_{\geqslant 1}^R(E,X)=0$
and $\Hom_{R[t]/(t^n)}((3.6), T(E)^+)$ is exact by \cite[Lemma 3.2]{XYY}. Since $M$ is Gorenstein flat in $\Mod R$,
$X$ has finite Gorenstein flat dimension. Because the subcategory of $\Mod R$ consisting of Gorenstein flat modules
is closed under extensions by \cite[Theorem 3.11]{SS},
it follows from \cite[Theorem 2.8]{Be} that $X$ is Gorenstein flat in $\Mod R$. Repeating this process, we may construct an exact sequence
$${\bf F''}: 0\to X \to F_0''\to F_1''\to \cdots $$
in $\Mod(R[t]/(t^n))$ with all $F''_i$ flat such that $\Hom_{R[t]/(t^n)}({\bf F''}, T(E)^+)$ is exact for any injective right $R$-module $E$.
Consequently, $M$ is Gorenstein flat in $\Mod(R[t]/(t^n))$.
\end{proof}

If $R$ is a commutative Noetherian ring, it is derived from \cite[Corollary 2.17]{H1} that $\Gfd_{R[t]/(t^2)}M=\Gfd_{R}M$
for any $R$-module $M$. Now, we will generalize this result to a more general setting by applying Theorem ~\ref{thm: 3.10}.

\begin{corollary}\label{cor: 3.11}
Let $R$ be a ring. Then for any $M\in \Mod(R[t]/(t^n))$, we have
$$\Gfd_{R[t]/(t^n)}M=\Gfd_{R}M.$$
\end{corollary}

\subsection{Homological conjectures}

\begin{lemma}\label{lem: 3.12}
Let $R$ be a ring and $M\in \Mod R$. If $f: M\to E^0(M)$ is the injective envelope of $M$ in $\Mod R$,
then $T(f): T(M)\to T(E^0(M))$ is the injective envelope of $T(M)$ in $\Mod (R[t]/(t^n))$.
\end{lemma}

\begin{proof}
Let $f: M\to E^0(M)$ be the injective envelope of $M$ in $\Mod R$. It follows from Proposition ~\ref{prop: 3.6} that
$T(E^0(M))$ is injective in $\Mod (R[t]/(t^n))$. Now let
$g\in \Hom_{\Mod R[\epsilon]^n}(T(E^0(M)),T(E^0(M))$ such that $gT(f)=T(f)$.
Since $g\epsilon_{T(E^0(M))}=\epsilon_{T(E^0(M))}g$, by the proof of Proposition ~\ref{prop: 3.4},
we may assume that $g$ has the following form
$$g=\left(
\begin{array}{ccccc}
a_1& 0& 0& \cdots &0 \\
a_2& a_1& 0& \cdots &0 \\
a_{3}& a_2& a_1& \cdots &0 \\
\vdots&\vdots &\vdots &\ddots&\vdots\\
a_{n}& a_{n-1}&a_{n-2}& \cdots &a_1\\
\end{array}
\right).$$
The equation $gT(f)=T(f)$ gives that $a_1f=f$. As $f$ is the injective envelope of $M$, $a_1$ is an isomorphism.
So $g$ is also an isomorphism. It implies that $T(f): T(M)\to T(E^0(M))$ is the injective envelope of $T(M)$
in $\Mod (R[t]/(t^n))$.
\end{proof}

In the rest of this subsection, $R$ is an Artinian algebra and
\begin{equation}\label{diag: (3.7)}
\begin{gathered}
0\to {_RR}\to E^0(R)\to E^1(R)\to \cdots \to E^i(R) \to \cdots
\end{gathered}
\end{equation} is a minimal injective resolution of $R$ in $\mod R$.
By Lemma ~\ref{lem: 3.12}, we immediately get a minimal injective resolution
\begin{equation}\label{diag: (3.8)}
\begin{gathered}
0\to T(R)\to T(E^0(R))\to T(E^1(R))\to \cdots \to T(E^i(R)) \to \cdots
\end{gathered}
\end{equation}
of $T(R)$ in $\mod (R[t]/(t^n))$. For a module $M\in\mod R$, we use $\pd_RM$ and $\id_RM$
to denote the projective and injective dimensions of $M$ respectively.
The following are some long-standing homological conjectures.
\begin{enumerate}
\item Finitistic Dimension Conjecture ({\bf FDC}) \cite{Ba1}: $\findim R:=\{\pd_RM\mid M\in \mod R$ with $\pd_RM<\infty\}<\infty$.
\item Strong Nakayama Conjecture ({\bf SNC}) \cite{CF}:  For any $0\neq M\in\mod R$, there exists $n\geqslant 0$ such that $\Ext^n_R(M,R)\neq 0$.
\item Generalied Nakayama Conjecture ({\bf GNC}) \cite{AR2}: Any indecomposable injective module in $\mod R$ occurs as a direct summand of some $E^i(R)$.
\item Auslander-Gorenstein Conjecture ({\bf AGC}) \cite{AR4}: If $R$ satisfies the Auslander condition (that is, $\pd_RE^i(R)\leqslant i$
for any $i\geqslant 0$), then $R$ is Gorenstein.
\item Nakayama Conjecture ({\bf NC}) \cite{N}: If $E^i(R)$ is projective for any $i\geqslant 0$, then $R$ is self-injective.
\item Gorenstein Symmetric Conjecture ({\bf GSC}) \cite{AR3}: $\id_{R}R<\infty$ if and only if $\id_{R^{op}}R<\infty$
(equivalently, $\id_{R}R=\id_{R^{op}}R$ by \cite[Lemma A]{Z}).
\end{enumerate}

Auslander and Reiten posed many conjectures, but they did not name the fourth conjecture above. For the sake of avoiding confusion and convenience,
we name it as Auslander-Gorenstein Conjecture. In general, we have the following implications:
$$\xymatrix{ {\bf FDC}  \ar@{=>}[r] \ar@{=>}[d] & {\bf SNC} \ar@{=>}[r] & {\bf GNC} \ar@{=>}[r]
& {\bf AGC} \ar@{=>}[r] & {\bf NC}  \\
{\bf GSC}. &  }$$

By \cite[Proposition 6.10]{AR3} and \cite[Theorem 3.4.3]{Y}, we have ${\bf FDC}\Rightarrow {\bf GSC}$ and
${\bf FDC}\Rightarrow {\bf SNC} \Rightarrow {\bf GNC}$ respectively. It is easy to see that {\bf NC} is a special case of {\bf AGC}.
Assume that $R$ satisfies the Auslander condition. Let $\{Q_1,\cdots,Q_s\}$ be a complete set of non-isomorphic indecomposable injective
modules in $\mod R$. If $R$ satisfies {\bf GNC}, then each $Q_i$ occurs as a direct summand of some $E^{t_i}(R)$. Set $m:=max\{t_1,\cdots,t_s\}$.
Then $\pd_RQ_i\leq m$ for any $1\leqslant i \leqslant s$. Thus the projective dimension of any injective module in $\mod R$ is at most $m$.
It follows that the injective dimension of any projective module in $\mod R^{op}$ is also at most $m$, in particular, $\id_{R^{op}}R\leqslant m$.
So $R$ is $m$-Gorenstein by \cite[Corollary 5.5(b)]{AR4}. This proves ${\bf GNC}\Rightarrow {\bf AGC}$.

\begin{theorem}\label{thm: 3.13}
\begin{enumerate}
\item[]
\item $R$ satisfies {\bf FDC} if and only if $R[t]/(t^n)$ satisfies {\bf FDC}.
\item $R$ satisfies {\bf SNC} if and only if $R[t]/(t^n)$ satisfies {\bf SNC}.
\item $R$ satisfies {\bf GNC} if and only if $R[t]/(t^n)$ satisfies {\bf GNC}.
\item $R$ satisfies {\bf AGC} if and only if $R[t]/(t^n)$ satisfies {\bf AGC}.
\item $R$ satisfies {\bf NC} if and only if $R[t]/(t^n)$ satisfies {\bf NC}.
\item $R$ satisfies {\bf GSC} if and only if $R[t]/(t^n)$ satisfies {\bf GSC}.
\end{enumerate}
\end{theorem}

\begin{proof}
Note that $R[t]/(t^n)$ as a left $R[t]/(t^n)$-module is isomorphic to the projective object $T(R)$ in $\Mod (R[t]/(t^n))$.
By Corollary \ref{cor: 3.8}(5), we have that $R$ is an Artinian algebra if and only if so is $R[t]/(t^n)$.
In addition, we always treat a left $R[t]/(t^n)$-module $A$ as an $n$-th differential object $(A, \epsilon_A)$.

(1) Suppose that $R$ satisfies {\bf FDC} and $\findim R=n(<\infty)$. Let $A\in\mod(R[t]/(t^n))$ with
$\pd_{R[t]/(t^n)}A<\infty$. By Proposition ~\ref{prop: 3.6}, we have the following projective resolution
$$0\to T(P_m)\to \cdots \to T(P_1)\to T(P_0)\to (A, \epsilon_A)\to 0$$
of $(A, \epsilon_A)$ in $\mod(R[t]/(t^n))$ such that each $P_i$ is projective left $R$-module.
Indeed, the above resolution is also a projective resolution of $A$ as a left $R$-module. Thus we have
$m\leqslant n$ and $\findim R[t]/(t^n)\leqslant n$.

Conversely, suppose that $R[t]/(t^n)$ satisfies {\bf FDC} and $\findim R[t]/(t^n)=n(<\infty)$. Let $A\in\mod R$ with $\pd_{R}A<\infty$.
Thus there exists a projective resolution
$$0\to P_m\to \cdots \to P_1\to P_0\to A\to 0$$
of $A$ in $\mod R$. Applying the exact functor $T$ to it yields a projective resolution
$$0\to T(P_m)\to \cdots \to T(P_1)\to T(P_0)\to T(A)\to 0$$
of $T(A)$ in $\mod(R[t]/(t^n))$. Thus we have $m\leqslant n$ and $\findim R\leqslant n$.

(2) Suppose that $R$ satisfies {\bf SNC}. Let $0\neq A\in \mod (R[t]/(t^n))$. Then there exists $n\geqslant 0$
such that $\Ext^n_R(A,R)\neq 0$. If $n\geqslant 1$, then by \cite[Theorem 3.9]{XYY}, we have $\Ext^n_{R[t]/(t^n)}(A,R[t]/(t^n))\neq 0$.
If $n=0$, then there exists $0\neq f\in \Hom_R(A,R)$. Thus
$$0\neq\left(
\begin{array}{c}
f\epsilon_A^{n-1} \\
\vdots \\
f\epsilon_A \\
f\\
\end{array}
\right) \in \Hom_{\Mod R[\epsilon]^n}(A,TR)$$
and $\Hom_{R[t]/(t^n)}(A, R[t]/(t^n))\neq 0$.

Conversely, suppose that $R[t]/(t^n)$ satisfies {\bf SNC}. Let $0\neq A \in \mod R$. Then $0\neq T(A) \in \mod (R[t]/(t^n))$ and
there exists $n\geqslant 0$ such that $\Ext^n_{R[t]/(t^n)}(T(A),R[t]/(t^n))$ $\neq 0$. If $n\geqslant 1$, then by \cite[Theorem 3.9]{XYY},
we have $\Ext^n_R(A,R)\neq 0$. For the case $n=0$, it is trivial that $\Hom_R(A,R)\neq 0$.

(3) Suppose that $R$ satisfies {\bf GNC}. Let $I$ be an indecomposable injective left $R[t]/(t^n)$-module.
Then $I\cong T(E)$ for some indecomposable injective left $R$-module $E$ by Proposition ~\ref{prop: 3.6}.
Since $E$ is isomorphic to a direct summand of some $E^i(R)$ by assumption, we have that $T(E)$ is isomorphic to
a direct summand of $T(E_i(R))$.

Conversely, suppose that $R[t]/(t^n)$ satisfies {\bf GNC}. Let $E$ be an indecomposable injective left $R$-module.
Then $T(E)$ is an indecomposable injective left $R[t]/(t^n)$-module. Since $T(E)$ is isomorphic to a direct summand
of some $T(E^i(R))$ by assumption, we have that $E$ is isomorphic to a direct summand of $E^i(R)$.

(4) Suppose that $R$ satisfies {\bf AGC}. If $\pd_{R[t]/(t^n)}T(E^i(R))\leqslant i$ for any $i\geqslant 0$,
it follows from Proposition ~\ref{prop: 3.6} that $\pd_{R}T(E^i(R))\leqslant i$ for any $i\geqslant 0$.
Hence $\pd_{R}E^i(R)\leqslant i$ for any $i\geqslant 0$. Since $R$ is Gorenstein by assumption,
we have that $R[t]/(t^n)$ is Gorenstein as well by Corollary ~\ref{cor: 3.8}(2).

Conversely, suppose that $R[t]/(t^n)$ satisfies {\bf AGC}. If $\pd_{R}E^i(R)\leqslant i$ for any $i\geqslant 0$,
then $\pd_{R[t]/(t^n)}T(E^i(R))\leqslant i$ for any $i\geqslant 0$. Since $R[t]/(t^n)$ is Gorenstein by assumption,
we have that $R$ is Gorenstein by Corollary ~\ref{cor: 3.8}(2) again.

(5) Suppose that $R$ satisfies {\bf NC}. If $T(E^i(R))$ is a projective left $R[t]/(t^n)$-module for any $i\geqslant 0$,
then in light of Proposition ~\ref{prop: 3.6}(1), we have that $E^i(R)$ is a projective left $R$-module for any $i\geqslant 0$.
By assumption, $R$ is self-injective. Then $R[t]/(t^n)$ is also self-injective by Corollary ~\ref{cor: 3.8}(2).

Conversely, suppose that $R[t]/(t^n)$ satisfies {\bf NC}. If $E^i(R)$ is a projective left $R$-module for any $i\geqslant 0$,
then $T(E^i(R))$ is a projective left $R[t]/(t^n)$-module for any $i\geqslant 0$. By assumption, $R[t]/(t^n)$ is self-injective.
Then $R$ is also self-injective by Corollary ~\ref{cor: 3.8}(2) again.

(6) It follows directly from Corollary ~\ref{cor: 3.8}(2).
\end{proof}

The results from Corollary \ref{cor: 3.8} to Theorem \ref{thm: 3.13} show that $R$ and $R[t]/(t^n)$ have many homological properties in common.
However, it is not always true. For example, let $K$ be a field. Then the global dimension of $K$ is zero, but $K[t]/(t^n)$ (where $n\geqslant 2$)
is a self-injective Nakayama algebra with infinite global dimension (\cite[Chapter V]{ASS}).

\section{\bf Triangulated categories}

In this section, we will introduce the homotopy category of $\mathcal{C}[\epsilon]^n$ and study how this
homotopy category is equivalent to the stable category of a Frobenius category.

\begin{lemma}\label{lem: 4.1}
Let $(\mathcal{C},\mathscr{E})$ be an idempotent complete exact category. Then $(\mathcal{C},\mathscr{E})$
is a Frobenius category if and only if $(\mathcal{C}[\epsilon]^n,\mathscr{E}_F)$ is a Frobenius category.
\end{lemma}

\begin{proof}
Let $(\mathcal{C},\mathscr{E})$ be a Frobenius category and $(X,\epsilon_X)\in \ob\mathcal{C}[\epsilon]^n$.
By Lemmas \ref{lem: 3.2}(1) and ~\ref{lem: 3.5}, $(\mathcal{C}[\epsilon]^n,\mathscr{E}_F)$ is an exact category
and there exists an admissible epic $TX\buildrel {p_X'} \over \rightarrow X$ in $\mathcal{C}[\epsilon]^n$.
Since $(\mathcal{C},\mathscr{E})$ is a Frobenius category, there exists an admissible epic $P\buildrel {\pi} \over \rightarrow X$
in $\mathcal{C}$ with $P$ projective. Thus we get an admissible epic
$T(P)\buildrel {p_X'T(\pi)} \over \longrightarrow X$ in $\mathcal{C}[\epsilon]^n$. It implies that
$(\mathcal{C}[\epsilon]^n,\mathscr{E}_F)$ has enough projectives. Dually, $(\mathcal{C}[\epsilon]^n,\mathscr{E}_F)$ has enough injectives.
Finally, an application of Proposition ~\ref{prop: 3.6} gives that the projectives and injectives in $\mathcal{C}[\epsilon]^n$
coincide. So $(\mathcal{C}[\epsilon]^n,\mathscr{E}_F)$ is a Frobenius category.

Conversely, let $(\mathcal{C}[\epsilon]^n,\mathscr{E}_F)$ be a Frobenius category and $M\in \ob\mathcal{C}$.
Then there exists an admissible epic $P\buildrel {p} \over \rightarrow TM$ in $\mathcal{C}[\epsilon]^n$ with $P$ projective.
It follows from \cite[Lemma 2.7]{B} that there exists an admissible epic $TM\buildrel {\pi} \over \rightarrow M$ in $\mathcal{C}$.
Thus we get an admissible epic $P\buildrel {\pi p} \over \rightarrow M$ in $\mathcal{C}$. It implies that
$\mathcal{C}$ has enough projectives. Dually, $\mathcal{C}$ has enough injectives. With the aid of Proposition ~\ref{prop: 3.6},
we obtain that the projectives and injectives in $\mathcal{C}$ coincide. Therefore $\mathcal{C}$ is a Frobenius category.
\end{proof}

When $\mathcal{C}$ is a Frobenius category, Happel showed that the stable category $\underline{\mathcal{C}}$
becomes a triangulated category (\cite[Chapter I, Section 2]{Ha}). In the following, we always assume that $\mathcal{C}$ is
an exact category with trivial exact structure $\mathscr{E}^t$ (that is, the short exact sequences are split exact sequences) and the
induced exact structure via the forgetful functor $F$ in $\mathcal{C}[\epsilon]^n$ is denoted by $\mathscr{E}_F^t$,
that is, a sequence $$0\to A\to B\to C\to 0$$ belongs to $\mathscr{E}_F^t$ when it splits in $\mathcal{C}$.

\begin{proposition}\label{prop: 4.2}
Let $(\mathcal{C},\mathscr{E}^t)$ be an idempotent complete exact category. Then $(\mathcal{C}[\epsilon]^n,\mathscr{E}_F^t)$
is a Frobenius category and $\underline{\mathcal{C}[\epsilon]^n}$ is a triangulated category.
\end{proposition}

\begin{proof}
By the definition of $\mathscr{E}_F^t$, every object in $\mathcal{C}$ is both projective and injective.
Thus $\mathcal{C}$ is a Frobenius category, and therefore $(\mathcal{C}[\epsilon]^n,\mathscr{E}_F^t)$ is
a Frobenius category by Lemma ~\ref{lem: 4.1}. Moreover we have that $\underline{\mathcal{C}[\epsilon]^n}$
is a triangulated category by \cite[Chapter I, Section 2]{Ha}.
\end{proof}

Recall that a {\it model structure} on a category $\mathcal{C}$
is three subcategories of $\mathcal{C}$ called {\it weak equivalences},
{\it cofibrations} and {\it fibrations} that must satisfy some axioms, see \cite[Definition 1.1.3]{Ho} for details.
Next we recall some notions from \cite{G}.
Given an exact category $(\mathcal{C},\mathscr{E})$,  by a {\it thick subcategory} of $\mathcal{C}$ we mean a class
of objects $\mathcal{W}$ which is closed under direct summands and such that if two out of three of the terms in a
short exact sequence are in $\mathcal{W}$, then so is the third. Suppose that $(\mathcal{C},\mathscr{E})$ has a model structure.
For an object $X\in \mathcal{C}$, we say that $X$ is {\it trivial} if $0\to X$ is a weak equivalence, $X$ is {\it cofibrant}
if $0\to X$ is a cofibration, and $X$ is {\it fibrant} if $X\to 0$ is a fibration. Moreover, we say $X$ is {\it trivially cofibrant}
if it is both trivial and cofibrant, and $X$ is {\it trivially fibrant} if it is both trivial and fibrant.

\begin{definition}\label{df: 4.3}(\cite{G})
Let $(\mathcal{C},\mathscr{E})$ be an exact category.
An {\it exact model structure} on $(\mathcal{C},\mathscr{E})$ is a model structure in which each of the following holds.
\begin{enumerate}
\item A map is a (trivial) cofibration if and only if it is an admissible monic with a (trivially) cofibrant cokernel.
\item A map is a (trivial) fibration if and only if it is an admissible epic with a (trivially) fibrant kernel.
\end{enumerate}
\end{definition}

The following corollary points out that the Frobenius category $(\mathcal{C}[\epsilon]^n,\mathscr{E}_F^t)$ has an exact model structure.

\begin{corollary}\label{cor: 4.4}
Let $(\mathcal{C},\mathscr{E}^t)$ be an idempotent complete exact category. Then there exists an exact model structure
on $(\mathcal{C}[\epsilon]^n,\mathscr{E}_F^t)$ given as follows.
\begin{enumerate}
\item A cofibration (resp. trivial cofibration) is an admissible monomorphism (resp. with a cokernel projective).
\item A fibration (resp. trivial fibration) is an admissible epimorphism (resp. with a kernel injective).
\item The weak equivalences are then the maps $g$ which factor as $g = pi$ where $i$ is a trivial
cofibration and $p$ is a trivial fibration.
\end{enumerate}
\end{corollary}

\begin{proof}
First of all, it follows from Lemma ~\ref{lem: 3.3} that $(\mathcal{C}[\epsilon]^n,\mathscr{E}_F^t)$
is an idempotent complete exact category. Since $(\mathcal{C}[\epsilon]^n,\mathscr{E}_F^t)$ is a
Frobenius category by Proposition  ~\ref{prop: 4.2}, the class $\mathcal{W}$ of all projective (=injective)
objects forms a thick subcategory of $\mathcal{C}[\epsilon]^n$. It is trivial that
$(\mathcal{W},\mathcal{C}[\epsilon]^n)$ and $(\mathcal{C}[\epsilon]^n, \mathcal{W})$
are complete cotorsion pairs. Then the assertions hold by \cite[Theorem 3.3 and Corollary 3.4]{G}.
\end{proof}

Now we are able to give an explicit description of the translation functor $\Sigma$ in $\underline{\mathcal{C}[\epsilon]^n}$.

\begin{corollary}\label{cor: 4.5}
Let $(\mathcal{C},\mathscr{E}^t)$ be an idempotent complete exact category, and let $X=(X, \epsilon_X)$
and $Y=(Y,\epsilon_Y)$ be objects in $\underline{\mathcal{C}[\epsilon]^n}$. Then we have
\begin{enumerate}
\item $\Sigma X=(X'', \epsilon_{X''})$, where $X''=X^{\oplus n-1}$ and
$$\epsilon_{X''}=\left(\begin{array}{ccccc}
0& 1&  0 &\cdots & 0 \\
0& 0&  1 &\cdots & 0 \\
\vdots& \vdots& \vdots&\ddots  &\vdots\\
0& 0&  0 &\cdots & 1 \\
-{\epsilon_X}^{n-1}& -{\epsilon_X}^{n-2}& -{\epsilon_X}^{n-3}& \cdots & -{\epsilon_X} \\
\end{array}
\right)_{(n-1)\times (n-1)}.$$
\item If $\underline{f}: X\to Y$, then $\Sigma \underline{f}=\underline{g}$, where
$$g=\left(\begin{array}{cccc}
f& 0& \cdots & 0 \\
0& f& \cdots & 0 \\
\vdots& \vdots& \ddots &\vdots\\
0& 0&  \cdots & f \\
\end{array}
\right)_{(n-1)\times (n-1)}$$
and the standard triangle associated to $\underline{f}$ is $$X\buildrel {\underline{f}} \over \longrightarrow
Y \buildrel {\underline{u}} \over \longrightarrow \Cone(f) \buildrel {\underline{v}} \over \longrightarrow \Sigma X$$
with $$\Cone(f)=Y\oplus X^{\oplus n-1}, \epsilon_{\Cone(f)}= \left(
\begin{array}{ccccc}
\epsilon_Y& f& 0&\cdots &0 \\
0& 0& 1&\cdots &0 \\
\vdots&\vdots &\vdots &\ddots&\vdots\\
0& 0& 0&\cdots &1 \\
0& -{\epsilon_X}^{n-1}& -{\epsilon_X}^{n-2}& \cdots&-{\epsilon_X}\\
\end{array}
\right)_{n\times n},$$
$$u=\left(\begin{array}{c}
1 \\
0\\
\vdots\\
0\\
\end{array}
\right), \,\,v=\left(\begin{array}{ccccc}
0& 1& \cdots & 0 & 0\\
0& 0& \cdots & 0 & 0\\
\vdots& \vdots& \ddots &\vdots &\vdots\\
0& 0&  \cdots & 0 &1 \\
\end{array}
\right)_{(n-1)\times n}.
$$
\end{enumerate}
\end{corollary}

\begin{proof}
(1) In view of Proposition ~\ref{prop: 4.2}, $\underline{\mathcal{C}[\epsilon]^n}$ is a triangulated category.
We also know from Lemma ~\ref{lem: 3.5} that there exists a short exact sequence
$$X\buildrel {i''_X} \over \longrightarrow  TX \buildrel {p''_X} \over \longrightarrow X''$$
in $\mathcal{C}[\epsilon]^n$. Since $\mathcal{C}$ is an exact category with trivial exact structure,
every object in $\mathcal{C}$ is injective. By Proposition \ref{prop: 3.6}(2),  we have that
$TX$ injective. Then one easily has $\Sigma X=(X'', \epsilon_{X''})$ by \cite[Chapter I, Section 2.2]{Ha}.

(2) For $\underline{f}: X\to Y$, we have $f\epsilon_X=\epsilon_Y f$. Then there exists a commutative diagram
$$\xymatrix{
 X \ar[r]^{i_X''} \ar[d]^{f}&  TX \ar[d]^h \ar[r]^{p_X''} &  X'' \ar[d]^{g}  \\
Y \ar[r]^{i_Y''} &  TY  \ar[r]^{p_Y''} &  Y''}$$ in $\mathcal{C}[\epsilon]^n$ with
$$h=\left(\begin{array}{cccc}
f& 0& \cdots & 0 \\
0& f& \cdots & 0 \\
\vdots& \vdots& \ddots &\vdots\\
0& 0&  \cdots & f \\
\end{array}
\right)_{n\times n},
g=\left(\begin{array}{cccc}
f& 0& \cdots & 0 \\
0& f& \cdots & 0 \\
\vdots& \vdots& \ddots &\vdots\\
0& 0&  \cdots & f \\
\end{array}
\right)_{(n-1)\times (n-1)}.$$
So $\Sigma \underline{f}=\underline{g}$. By \cite[Chapter I, Section 2.5]{Ha}, the standard triangle is constructed by the following push-out diagram
$$\xymatrix{
 X \ar[r]^{i_X''} \ar[d]^{f}&  TX \ar[d]^{h'} \ar[r]^{p_X''} &  \Sigma X \ar@{=}[d]  \\
Y \ar[r]^{u} &  \Cone(f)  \ar[r]^{v} &   \Sigma X. }$$ By the proof of Lemma ~\ref{lem: 3.2},
it suffices to construct a push-out along with $f$ and $i_X''$ in $\mathcal{C}$. Take
$$h'=\left(\begin{array}{ccccc}
0 & 0& \cdots &0 & f \\
0 & 0& \cdots &1 & -\epsilon_X \\
\vdots& \vdots& \iddots& \vdots& \vdots\\
0 & 1& \cdots &0 & 0 \\
1 & -\epsilon_X & \cdots &0 & 0\\
\end{array}
\right)_{n\times n}:TX\to Y\oplus X^{\oplus n-1}(=\Cone(f))$$
in $\mathcal{C}[\epsilon]^n$. It is easy to see that $h'i_X''=uf$. Now let $M\in \ob\mathcal{C}$,
$\alpha=(\alpha_1,\alpha_2,\cdots, \alpha_n): TX\to M$ and $\beta:Y\to M$ such that $\beta f=\alpha i_X''$.
We have to show that there exists a unique morphism $\gamma=(\gamma_1,\gamma_2,\cdots, \gamma_n): \Cone(f)\to M$
such that $\gamma u=\beta$ and $\gamma h=\alpha$. Let
$$\gamma_1=\beta, \gamma_n=\alpha_1, \gamma_i=\alpha_{n-i+1}+\alpha_{n-i}\epsilon_X+\cdots+\alpha_1{\epsilon_X}^{n-i}$$
for $2\leqslant i\leqslant n-1$. It is the morphism, as desired.
\end{proof}

\begin{definition}\label{df: 4.6}
A morphism $f: (X,\epsilon_X)\to (Y,\epsilon_Y)$ in $\mathcal{C}[\epsilon]^n$ is called
{\it null-homotopic} if there exists a morphism $s:X\to Y$ in $\mathcal{C}$ such that
$$f={\epsilon_{Y}}^{n-1}s+{\epsilon_{Y}}^{n-2}s{\epsilon_{X}}+\cdots + s{\epsilon_{X}}^{n-1}.$$
For morphisms $f,g: X\to Y$ in $\mathcal{C}[\epsilon]^n$, we denote $f\thicksim g$ if $f-g$
is null-homotopic. We denote by $\K(\mathcal{C}[\epsilon]^n)$ the {\it homotopy category},
that is, the category consisting of $n$-th differential objects such that the morphism set
between $X,Y\in \K(\mathcal{C}[\epsilon]^n)$ is given by $\Hom_{\K(\mathcal{C}[\epsilon]^n)}(X,Y)
=\Hom_{\mathcal{C}[\epsilon]^n}(X,Y)/\thicksim$.
\end{definition}

We close this section with the following theorem.

\begin{theorem}\label{thm: 4.7}
Let $(\mathcal{C},\mathscr{E}^t)$ be an idempotent complete exact category. Then the stable category
$\underline{\mathcal{C}[\epsilon]^n}$ of the Frobenius category $(\mathcal{C}[\epsilon]^n,\mathscr{E}_F^t)$
is the homotopy category $\K(\mathcal{C}[\epsilon]^n)$.
\end{theorem}

\begin{proof}
It suffices to show that a morphism $f: X \to Y$ in $\mathcal{C}[\epsilon]^n$ is null-homotopic if and only if
it factors through a projective object in $\mathcal{C}[\epsilon]^n$. Assume that $f : X \to Y$ is null-homotopic.
By definition, there exists a morphism $s:X\to Y$ in $\mathcal{C}$ such that
$$f={\epsilon_{Y}}^{n-1}s+{\epsilon_{Y}}^{n-2}s{\epsilon_{X}}+\cdots + s{\epsilon_{X}}^{n-1}.$$ Take
$$g=\left(
\begin{array}{c}
s{\epsilon_{X}}^{n-1} \\
\vdots \\
s\epsilon_{X} \\
s\\
\end{array}
\right): X\to TY.$$
Then $g$ is morphism in $\mathcal{C}[\epsilon]^n$ and $f=p_Y'g$. Thus $f$ factors through a projective object
since $TY$ is a projective object of $\mathcal{C}[\epsilon]^n$. Now suppose that $f$ factors through a projective object.
Then $f$ must factor through $TY$ and thus there exists a morphism
$$g=\left(
\begin{array}{c}
g_1 \\
g_2 \\
\vdots \\
g_n\\
\end{array}
\right): X\to TY$$ in $\mathcal{C}[\epsilon]^n$ such that
$$f=p_Y'g=g_1+\epsilon_Yg_2+\cdots +{\epsilon_{Y}}^{n-1}g_n.$$
Since $\epsilon_{Y^{\oplus n}}g=g\epsilon_X$, we have $g_1\epsilon_X=0$ and $g_{i+1}\epsilon_X=g_i$
for any $1\leqslant i \leqslant n-1$. Set $s:=g_n$. It follows that
$$f={\epsilon_{Y}}^{n-1}s+{\epsilon_{Y}}^{n-2}s{\epsilon_{X}}+\cdots + s{\epsilon_{X}}^{n-1}$$
and $f$ is null-homotopic.
\end{proof}

\section{\bf The derived category}
In this section, $\mathcal{A}$ is an abelian category. We will introduce the derived category of
$\mathcal{A}[\epsilon]^n$ as the Verdier quotient of the homotopy category $\K(\mathcal{A}[\epsilon]^n)$
with respect to quasi-isomorphisms.

A sequence
$$0\to X\to Y\to Z\to 0$$ in $\mathcal{A}[\epsilon]^n$ is exact if and only if
$$0\to FX\to FY\to FZ\to 0$$ is exact in $\mathcal{A}$. $\mathcal{A}[\epsilon]^n$ also forms an abelian category.
The next definition essentially generalizes the notion of homology used in \cite{S1}.

\begin{definition}\label{df: 5.1} We call $(X, \epsilon_X)\in \mathcal{A}[\epsilon]^n$ {\it acyclic} if
$$\H_{(r)}(X):= \Ker {\epsilon_X}^r/\im {\epsilon_X}^{n-r}=0$$ for any $1\leqslant r\leqslant n-1$.
\end{definition}

By the definition above, one easily see that any object $(X^{\oplus n}, \epsilon_{X^{\oplus n}})$ is acyclic.

\begin{proposition} \label{prop: 5.2}
Let $(X, \epsilon_X), (Y, \epsilon_Y) \in \mathcal{A}[\epsilon]^n$ and $f, g \in \Hom_{\mathcal{A}[\epsilon]^n}(X,Y)$.
If $f\thicksim g$, then $\H_{(r)}(f)=\H_{(r)}(g)$ for any $1\leqslant r\leqslant n-1$.
\end{proposition}

\begin{proof}
If $f\thicksim g$, then there exists a morphism $s: X\to Y$ such that
$$f-g={\epsilon_{Y}}^{n-1}s+{\epsilon_{Y}}^{n-2}s{\epsilon_{X}}+\cdots + s{\epsilon_{X}}^{n-1}.$$
So for any $1\leqslant r\leqslant n-1$, we have
$$(f-g)(\Ker {\epsilon_X}^r)={\epsilon_{Y}}^{n-1}s(\Ker {\epsilon_X}^r)+
{\epsilon_{Y}}^{n-2}s{\epsilon_{X}}(\Ker ({\epsilon_X}^r))+\cdots$$
$$+{\epsilon_{Y}}^{n-r+1}s {\epsilon_X}^{r-1}(\Ker {\epsilon_X}^r).$$
Thus $(f-g)(\Ker {\epsilon_X}^r)\subseteq \im {\epsilon_{Y}}^{n-r}$, and therefore $\H_{(r)}(f)=\H_{(r)}(g)$.
\end{proof}

\begin{lemma}\label{lem: 5.3}
Let $$0\to X\buildrel {f} \over \longrightarrow Y \buildrel {g} \over \longrightarrow Z\to 0$$
be an exact sequence in $\mathcal{A}[\epsilon]^n$. Then we have the following exact sequence
$$\cdots \buildrel {\partial} \over \longrightarrow \H_{(r)}(X)\buildrel {} \over \longrightarrow
\H_{(r)}(Y)\buildrel {} \over \longrightarrow \H_{(r)}(Z)\buildrel {\partial} \over \longrightarrow
\H_{(n-r)}(X)\buildrel {} \over \longrightarrow \H_{(n-r)}(Y)\buildrel {} \over \longrightarrow \cdots.$$
\end{lemma}

\begin{proof}
For any $n$-th differential object $(X, \epsilon_X)$, we may construct a complex
$$\cdots \buildrel {{\epsilon_X}^{n-r}} \over \longrightarrow X \buildrel {{\epsilon_X}^r} \over \longrightarrow
X \buildrel {{\epsilon_X}^{n-r}} \over \longrightarrow \cdots.$$
Consider the following diagram
$$\xymatrix{ & \vdots \ar[d]^{{\epsilon_X}^{r}} & \vdots \ar[d]^{{\epsilon_Y}^{r}} &  \vdots \ar[d]^{{\epsilon_Z}^{r}} \\
 0 \ar[r] & X \ar[r]^f \ar[d]^{{\epsilon_X}^{n-r}}& Y \ar[d]^{{\epsilon_Y}^{n-r}} \ar[r]^g & Z \ar[d]^{{\epsilon_Z}^{n-r}} \ar[r]& 0 \\
 0 \ar[r] & X \ar[r]^f \ar[d]^{{\epsilon_X}^{r}}& Y \ar[d]^{{\epsilon_Y}^{r}} \ar[r]^g & Z \ar[d]^{{\epsilon_Z}^{r}} \ar[r] & 0 \\
 0 \ar[r] & X \ar[r]^f \ar[d]^{{\epsilon_X}^{n-r}}& Y \ar[d]^{{\epsilon_Y}^{n-r}} \ar[r]^g & Z \ar[d]^{{\epsilon_Z}^{n-r}} \ar[r] & 0 \\
 & \vdots & \vdots & \vdots}$$ in $\mathcal{A}$. Then the desired exact sequence follows from \cite[Theorem 1.3.1]{We}.
\end{proof}

We use $\K^a(\mathcal{A}[\epsilon]^n)$ to denote the full subcategory of $\K(\mathcal{A}[\epsilon]^n)$
consisting of all acyclic objects.

\begin{proposition} \label{prop: 5.4}
$\K^a(\mathcal{A}[\epsilon]^n)$ is a thick triangulated subcategory of $\K(\mathcal{A}[\epsilon]^n)$.
\end{proposition}

\begin{proof}
By Corollary ~\ref{cor: 4.5}, there exists an exact sequence
$$0\to X\to TX\to \Sigma X\to 0$$ in $\mathcal{A}[\epsilon]^n$. Note that $TX$ is always acyclic.
If $X$ is acyclic, then $\Sigma X$ is also acyclic by Lemma ~\ref{lem: 5.3}. It implies that
$\K^a(\mathcal{A}[\epsilon]^n)$ is closed under $\Sigma$. Dually, $\K^a(\mathcal{A}[\epsilon]^n)$
is closed under $\Sigma^{-1}$. Let
$$X\buildrel {\underline{f}} \over \longrightarrow Y\buildrel {\underline{g}} \over \longrightarrow
Z\buildrel {\underline{h}} \over \longrightarrow \Sigma X$$ be a triangle in $\K(\mathcal{A}[\epsilon]^n)$
with $X$ and $Y$ acyclic. We have to prove that $Z$ is acyclic as well. By Corollary ~\ref{cor: 4.5},
$Z\cong \Cone(f)$ in $\K(\mathcal{A}[\epsilon]^n)$. It suffices to show that $\Cone(f)$ is acyclic by
Proposition ~\ref{prop: 5.2}. Indeed, we have the following commutative diagram of exact sequences
$$\xymatrix{
0 \ar[r] & X \ar[r]^{i_X''} \ar[d]^{f}&  TX \ar[d]^h \ar[r]^{p_X''} &  \Sigma X \ar@{=}[d] \ar[r] & 0\\
0 \ar[r] & Y \ar[r]^{g} &  \Cone(f)  \ar[r]^{h} &   \Sigma X \ar[r] & 0}$$
in $\mathcal{A}[\epsilon]^n$. Then we get an exact sequence
$$0\to X\to Y\oplus TX\to \Cone(f)\to 0$$ in $\mathcal{A}[\epsilon]^n$. Since $TX$ is acyclic,
we obtain that $\Cone(f)$ is also acyclic by Lemma ~\ref{lem: 5.3}. Obviously $\K^a(\mathcal{A}[\epsilon]^n)$
is closed under direct summands. The proof is finished.
\end{proof}

\begin{definition}\label{df: 5.5}
\begin{enumerate}
\item[]
\item A morphism $f: X \to Y$ of $\K(\mathcal{A}[\epsilon]^n)$ is called a {\it quasi-isomorphism} if
$\H_{(r)}(f) : \H_{(r)}(X) \to \H_{(r)}(Y)$ is an isomorphism for any $1\leqslant r \leqslant n-1$, or equivalently
by Lemma ~\ref{lem: 5.3}, $\Cone(f)$ is acyclic.
\item The {\it derived category} of
$n$-differential objects is defined as the quotient category
$$\D(\mathcal{A}[\epsilon]^n):=\K(\mathcal{A}[\epsilon]^n)/\K^a(\mathcal{A}[\epsilon]^n).$$
\end{enumerate}
\end{definition}

Actually, in view of Definition ~\ref{df: 4.6}, the homotopy category and derived category of $n$-differential
objects in $\mathcal{A}[\epsilon]^n$ differ from that of complexes in $\mathcal{A}$.

By definition, a morphism in $\K(\mathcal{A}[\epsilon]^n)$ is a quasi-isomorphism if and only if it is an
isomorphism in $\D(\mathcal{A}[\epsilon]^n)$.  Let us present the following definitions in order to
simplify some statements and notations.

\begin{definition}\label{df: 5.6}
\begin{enumerate}
\item[]
\item We say that $X\in \ob\K(\mathcal{A}[\epsilon]^n)$ is {\it $K$-projective} if $\Hom_{\K(\mathcal{A}[\epsilon]^n)}(X,Y)=0$
for any $Y\in \ob\K^a(\mathcal{A}[\epsilon]^n)$. Dually we say that $X\in \ob\K(\mathcal{A}[\epsilon]^n)$ is {\it $K$-injective}
if $\Hom_{\K(\mathcal{A}[\epsilon]^n)}(Y,X)=0$ for any $Y\in \ob\K^a(\mathcal{A}[\epsilon]^n)$. We denote by
$\K^p(\mathcal{A}[\epsilon]^n)$ (resp. $\K^i(\mathcal{A}[\epsilon]^n))$ the full subcategory of
$\K(\mathcal{A}[\epsilon]^n)$ consisting of $K$-projective (resp. $K$-injective) $n$-th differential objects.
\item Assume that $\mathcal{A}$ has enough projective and injective objects. A {\it projective resolution}
(resp. {\it injective resolution}) of $X\in \ob\K(\mathcal{A}[\epsilon]^n)$ is a quasi-isomorphism
$P_X \to X$ (resp. $X \to I_X$) with $P_X\in \ob\K^p(\mathcal{A}[\epsilon]^n)$ and $F(P_X)$ projective
(resp. $I_X \in \ob\K^i(\mathcal{A}[\epsilon]^n)$ and $F(I_X)$ injective).
\end{enumerate}
\end{definition}

We have the following

\begin{proposition} \label{prop: 5.7}
\begin{enumerate}
\item[]
\item If $X$ is projective (resp. injective) in $\mathcal{A}$, then $(X,0)$ (resp. $(0,X))$ is $K$-projective (resp. $K$-injective).
\item $\K^p(\mathcal{A}[\epsilon]^n)$ and $\K^i(\mathcal{A}[\epsilon]^n)$ are triangulated subcategories of $\K(\mathcal{A}[\epsilon]^n)$.
\end{enumerate}
\end{proposition}

\begin{proof}
(1) Assume that $X$ is projective and $(Y,\epsilon_Y)$ is acyclic. Take $f$ to be a morphism from $(X,0)$ to $(Y,\epsilon_Y)$ in
$\mathcal{A}[\epsilon]^n$. Then we have $\epsilon_Yf=0$. Since $X$ is projective and the sequence
$$Y\buildrel {{\epsilon_Y}^{n-1}} \over \longrightarrow Y \buildrel {{\epsilon_Y}} \over \longrightarrow Y$$ is exact in $\mathcal{A}$,
there exists a morphism $s: X\to Y$ such that $f={\epsilon_Y}^{n-1}s$. Since $\epsilon_X=0$, we have
$$f={\epsilon_Y}^{n-1}s={\epsilon_{Y}}^{n-1}s+{\epsilon_{Y}}^{n-2}s{\epsilon_{X}}+\cdots + s{\epsilon_{X}}^{n-1}$$
and $(X,0)$ is $K$-projective. Dually, we get the other assertion.

(2) It is clear that $\K^p(\mathcal{A}[\epsilon]^n)$ is closed under isomorphisms and translation. Now assume that
$$X\to Y\to Z\to \Sigma X$$ is a triangle in $\K(\mathcal{A}[\epsilon]^n))$ with $X,Y\in \ob\K^p(\mathcal{A}[\epsilon]^n)$.
For any $M\in \ob\K^a(\mathcal{A}[\epsilon]^n)$, applying the functor $\Hom_{\K(\mathcal{A}[\epsilon]^n))}(-,M)$ yields
the following exact sequence
$$\Hom_{\K(\mathcal{A}[\epsilon]^n))}(\Sigma X,M)\to \Hom_{\K(\mathcal{A}[\epsilon]^n))}(Z,M)\to \Hom_{\K(\mathcal{A}[\epsilon]^n))}(Y,M).$$
The end terms vanish by assumption, hence the middle term also vanishes, which implies that $Z$ is $K$-projective.
We conclude that $\K^p(\mathcal{A}[\epsilon]^n)$ is a triangulated subcategory of $\K(\mathcal{A}[\epsilon]^n)$.
Dually, $\K^i(\mathcal{A}[\epsilon]^n)$ is also a triangulated subcategory of $\K(\mathcal{A}[\epsilon]^n)$.
\end{proof}

Recall from \cite{P} that an abelian category $\mathcal{A}$ is an {\it Ab4-category} (resp. {\it Ab4$^*$-category})
provided that it has an arbitrary coproduct (resp. product) of objects and the coproduct (resp. product) of
monomorphisms (resp., epimorphisms) is monic (resp. epic). The following lemma is crucial in proving Theorem ~\ref{thm: 5.11}.

\begin{lemma}\label{lem: 5.8}
\begin{enumerate}
\item[]
\item If $\mathcal{A}$ is an Ab4-category with enough projectives,
then any $X\in \ob\K(\mathcal{A}[\epsilon]^n)$ has a projective resolution.
\item If $\mathcal{A}$ is an Ab4$^*$-category with enough injectives,
then any $X\in \ob\K(\mathcal{A}[\epsilon]^n)$ has an injective resolution.
\end{enumerate}
\end{lemma}

\begin{proof}
(1) Let $X\in \ob\K(\mathcal{A}[\epsilon]^n)$. Then there exists a sequence
$${\bf X}:\cdots \buildrel {\epsilon_X} \over \longrightarrow X \buildrel {\epsilon_X} \over \longrightarrow
X \buildrel {\epsilon_X} \over \longrightarrow \cdots$$ with ${\epsilon_X}^n=0$. It is a special $N$-complex
in the language of \cite{IKM}. Then by the proof of \cite[Theorem 3.17]{IKM}, there exists an $N$-quasi-isomorphism
$s: {\bf P} \to {\bf X}$ as follows:
$$\xymatrix{
{\bf P}: \ar[d]^s & \cdots  \ar[r]^{d} & P \ar[r]^{d} \ar[d]^{s}
&  P \ar[r]^{d} \ar[d]^{s} &  P \ar[r]^{d} \ar[d]^{s} & \cdots \\
{\bf X}:  & \cdots \ar[r]^{\epsilon_X} &   X \ar[r]^{\epsilon_X}  &    X \ar[r]^{\epsilon_X}
& X \ar[r]^{\epsilon_X}  & \cdots}$$ with ${\bf P}\in \K^p_N(\mathcal{A})$ and $P^i$ projective in $\mathcal{A}$.
Thus $(P, d)$ is an $n$-th differential object
and $s: P\to X$ is a projective resolution of $X$.

(2) It is dual to (1).
\end{proof}

In order to demonstrate the key result in this section, we also need the following two definitions.

\begin{definition}\label{df: 5.9}(\cite{Mi})
Let $\mathcal{T}$ be a triangulated category. A pair $(\mathcal{U},\mathcal{V})$ of full triangulated subcategories
of $\mathcal{T}$ is called a {\it stable $t$-structure} in $\mathcal{T}$ provided that
$\Hom_{\mathcal{T}}(\mathcal{U},\mathcal{V})=0$ and
$\mathcal{T}=\mathcal{U}*\mathcal{V}:=\{t\in \mathcal{T}\mid$ there exists a triangle
$$u\to t\to v\to \Sigma u$$
with $u\in \mathcal{U}$ and $v\in \mathcal{V}\}$.
\end{definition}

\begin{definition}\label{df: 5.10}(\cite{BBD})
We call a diagram
$$\xymatrix{
\mathcal{D}'  \ar[r]^{i_*}& \mathcal{D} \ar@/^1pc/[l]_{i^!} \ar@/_1pc/[l]_{i^*}
\ar[r]^{j^*} & \mathcal{D}'' \ar@/^1pc/[l]_{j_*} \ar@/_1pc/[l]_{j_!} }$$
of triangulated categories and functors a {\it recollement} if the following conditions are satisfied.
\begin{enumerate}
\item $i_*$, $j_!$ and $j_*$ are fully faithful.
\item $(i^*,i_*)$, $(i_*,i^!)$, $(j_!,j^*)$ and $(j^*,j_*)$ are adjoint pairs.
\item There exist canonical embeddings $\im j_! \hookrightarrow \Ker i^*$, $\im i_* \hookrightarrow \Ker j^*$
and $\im j_* \hookrightarrow \Ker i^!$, which are equivalences.
\end{enumerate}
\end{definition}

Let $R$ be a ring and $\K(R)$ and $\D(R)$ its homotopy and derived categories respectively.
The kernel of the quotient functor $Q: \K(R)\to \D(R)$ is precisely the full subcategory $\K^a(R)$ of
all exact complexes (modulo the chain homotopy relation). The localization
$$\K^a(R) \buildrel {i} \over \longrightarrow \K(R)\buildrel {Q} \over \longrightarrow \D(R)$$
forms the center arrows in the following  recollement diagram (see \cite[Example 4.14]{Kr}).
$$\xymatrix{
\K^a(R)  \ar[r]^{i}& \K(R) \ar@/^1pc/[l] \ar@/_1pc/[l]\ar[r]^{Q} & \D(R). \ar@/^1pc/[l] \ar@/_1pc/[l]}$$
Inspired by this result, we will give the main theorem in this section.

\begin{theorem}\label{thm: 5.11}
\begin{enumerate}
\item[]
\item Assume that $\mathcal{A}$ is an Ab4-category with enough projectives.
Then we have a stable $t$-structure $(\K^p(\mathcal{A}[\epsilon]^n), \K^a(\mathcal{A}[\epsilon]^n))$
in $\K(\mathcal{A}[\epsilon]^n)$ and a triangle equivalence
$\K^p(\mathcal{A}[\epsilon]^n)\simeq \D(\mathcal{A}[\epsilon]^n)$.
\item Assume that $\mathcal{A}$ is an Ab4$^*$-category with enough injectives.
Then we have a stable $t$-structure $(\K^a(\mathcal{A}[\epsilon]^n), \K^i(\mathcal{A}[\epsilon]^n)$
in $\K(\mathcal{A}[\epsilon]^n)$ and a triangle equivalence
$\K^i(\mathcal{A}[\epsilon]^n)\simeq \D(\mathcal{A}[\epsilon]^n)$.
\item  Under the assumptions of (1) and (2), there exists a recollement
$$\xymatrix{
\K^a(\mathcal{A}[\epsilon]^n)  \ar[r]^{i_*}& \K(\mathcal{A}[\epsilon]^n)
\ar@/^1pc/[l]_{i^!} \ar@/_1pc/[l]_{i^*} \ar[r]^{j^*} & \D(\mathcal{A}[\epsilon]^n)
\ar@/^1pc/[l]_{j_*} \ar@/_1pc/[l]_{j_!} }.$$
\end{enumerate}
\end{theorem}

\begin{proof}
(1) It follows from Propositions ~\ref{prop: 5.4} and ~\ref{prop: 5.7} that both $\K^p(\mathcal{A}[\epsilon]^n)$
and $\K^a(\mathcal{A}[\epsilon]^n)$ are triangulated subcategories of $\K(\mathcal{A}[\epsilon]^n)$.
On the other hand, by Lemma ~\ref{lem: 5.8}(1), we have
$$\K(\mathcal{A}[\epsilon]^n)=\K^p(\mathcal{A}[\epsilon]^n)*\K^a(\mathcal{A}[\epsilon]^n).$$
Hence $(\K^p(\mathcal{A}[\epsilon]^n), \K^a(\mathcal{A}[\epsilon]^n))$ is a stable $t$-structure in $\K(\mathcal{A}[\epsilon]^n)$.
Furthermore, it is derived from \cite{JK} or \cite[Lemma 1.6]{IKM} that there exists a triangle equivalence
$\K^p(\mathcal{A}[\epsilon]^n)\simeq \D(\mathcal{A}[\epsilon]^n)$.

(2) It is dual to (1).

(3) By (1) and (2), both $(\K^p(\mathcal{A}[\epsilon]^n), \K^a(\mathcal{A}[\epsilon]^n))$ and
$(\K^a(\mathcal{A}[\epsilon]^n), \K^i(\mathcal{A}[\epsilon]^n)$ are stable $t$-structures.
Now the assertion follows from \cite[Proposition 1.8]{IKM0} (also cf. \cite{Mi}).
\end{proof}

\begin{remark}\label{rem: 5.12}
In general $\D(\mathcal{A}[\epsilon]^n)$ is not at all easy to understand. Even it is difficult to calculate
the morphisms in the derived category $\D(\mathcal{A}[\epsilon]^n)$. However, in particular cases, since
$\K^p(\mathcal{A}[\epsilon]^n)\simeq \D(\mathcal{A}[\epsilon]^n)$ by Theorem ~\ref{thm: 5.11} and $\K^p(\mathcal{A}[\epsilon]^n)$
is a full subcategory of $\K(\mathcal{A}[\epsilon]^n)$, the class of maps between two objects in $\D(\mathcal{A}[\epsilon]^n)$
actually forms a set, and those two triangle equivalences in Theorem ~\ref{thm: 5.11} provide easier ways
to represent morphisms in $\D(\mathcal{A}[\epsilon]^n)$.
\end{remark}

\begin{corollary}\label{cor: 5.13}
Assume that $\mathcal{A}$ is an Ab4-category with enough projectives.
Then both $\K(\mathcal{A}[\epsilon]^n)$ and $\D(\mathcal{A}[\epsilon]^n)$ have arbitrary coproducts.
\end{corollary}

\begin{proof}
We first show that $\K(\mathcal{A}[\epsilon]^n)$ has arbitrary coproducts. Let $\{X_i\}_{i\in I}$ be a family
of objects in $\K(\mathcal{A}[\epsilon]^n)$ and $Y$ an object in $\K(\mathcal{A}[\epsilon]^n)$.
By the proof of Theorem~\ref{thm: 4.7}, we get that a morphism $f: X \to Y$ is null-homotopic if and only if
it factors through $TY$. Then we have the following commutative diagram with exact rows
$$\tiny{\xymatrix{
\Hom_{\mathcal{A}[\epsilon]^n}(\coprod _{i\in I}X_i, TY) \ar[r] \ar[d]^{\cong}&
\Hom_{\mathcal{A}[\epsilon]^n}(\coprod _{i\in I}X_i, Y) \ar[r] \ar[d]^{\cong} &
\Hom_{\K(\mathcal{A}[\epsilon]^n)}(\coprod _{i\in I}X_i, Y) \ar[d] \ar[r]& 0 \\
\prod_{i\in I}\Hom_{\mathcal{A}[\epsilon]^n}(X_i, TY) \ar[r] &
\prod_{i\in I}\Hom_{\mathcal{A}[\epsilon]^n}(X_i, Y)  \ar[r] &
\prod_{i\in I}\Hom_{\mathcal{A}[\epsilon]^n}(X_i, Y)\ar[r] & 0.}}$$
It implies that
$$\Hom_{\K(\mathcal{A}[\epsilon]^n)}(\coprod _{i\in I}X_i, Y)\cong  \prod_{i\in I}\Hom_{\mathcal{A}[\epsilon]^n}(X_i, Y)$$
and arbitrary direct sums exist in $\K(\mathcal{A}[\epsilon]^n)$.

Next, since $\K^p(\mathcal{A}[\epsilon]^n)\simeq \D(\mathcal{A}[\epsilon]^n)$ by Theorem ~\ref{thm: 5.11}, it suffice to show that
$\K^p(\mathcal{A}[\epsilon]^n)$ has arbitrary coproducts. Let $\{X_i\}_{i\in I}\in \K^p(\mathcal{A}[\epsilon]^n)$
and $Y\in \K^a(\mathcal{A}[\epsilon]^n)$. Since
$$\Hom_{\K(\mathcal{A}[\epsilon]^n)}(\coprod _{i\in I}X_i,Y)\cong \prod_{i\in I}\Hom_{\K(\mathcal{A}[\epsilon]^n)}(X_i,Y)=0,$$
we have $\coprod _{i\in I}X_i\in \K^p(\mathcal{A}[\epsilon]^n)$.
\end{proof}

Given the fact that for any ring $R$, the derived category $\D(\Mod R)$ is always compactly generated. It is
natural to ask whether it is possible to get a similar result for $\D(\mathcal{A}[\epsilon]^n)$. To answer this question,
firstly let us recall the following definition.

\begin{definition}\label{df: 5.14}(\cite{St})
Let $\mathcal{T}$ be a triangulated category with arbitrary coproducts. An object $C\in \mathcal{T}$ is called {\it compact}
if for any family $\{Y_i\}_{i\in I}$ of objects of $\mathcal{T}$, the natural morphism
$$\coprod _{i\in I} \Hom_{\mathcal{T}}(C,Y_i)\to \Hom_{\mathcal{T}}(C,\coprod_{i\in I}Y_i)$$ is an isomorphism.
The category $\mathcal{T}$ is said to be {\it compactly generated} if there exists a set
$\mathcal{C}$ of compact objects satisfying the following property: if $X\in \mathcal{T}$ such that $\Hom_{\mathcal{T}}(C,X)=0$
for any $C\in \mathcal{C}$, then $X=0$.
\end{definition}

To state the last theorem of this section, we need the following two results.

\begin{lemma}\label{lem: 5.15}
Let $Y\in \K^p(\mathcal{A}[\epsilon]^n)$ and $f: X\to Y$ be a quasi-isomorphism in $\K(\mathcal{A}[\epsilon]^n)$.
Then there exists a morphism $g: Y\to X$ in $\mathcal{A}[\epsilon]^n$ such that $fg\thicksim 1_Y$.
\end{lemma}

\begin{proof}
Consider the triangle
$$X\buildrel {f} \over \longrightarrow Y\to \Cone(f) \to \Sigma X$$ in $\K^p(\mathcal{A}[\epsilon]^n)$.
Since $f$ is a quasi-isomorphism, $\Cone(f)$ is acyclic. By applying the functor $\Hom_{\K(\mathcal{A}[\epsilon]^n)}(Y,-)$
to this triangle, we get an exact sequence
$$\Hom_{\K(\mathcal{A}[\epsilon]^n)}(Y,X)\buildrel {\Hom_{\K(\mathcal{A}[\epsilon]^n)}(Y,f)}
\over \longrightarrow \Hom_{\K(\mathcal{A}[\epsilon]^n)}(Y,Y)\longrightarrow \Hom_{\K(\mathcal{A}[\epsilon]^n)}(Y,\Cone(f)).$$
As $Y\in \K^p(\mathcal{A}[\epsilon]^n)$, we have $\Hom_{\K(\mathcal{A}[\epsilon]^n)}(Y,\Cone(f))=0$.
Thus there exists a morphism $g: Y\to X$ in $\mathcal{A}[\epsilon]^n$ such that $fg\thicksim 1_Y$.
\end{proof}

\begin{proposition}\label{prop: 5.16}
Assume that $\mathcal{A}$ is an Ab4-category with enough projectives. Then for any $X\in \K^p(\mathcal{A}[\epsilon]^n)$ and
$Y\in \K(\mathcal{A}[\epsilon]^n)$, there exists an isomorphism of abelian groups
$$\Hom_{\K(\mathcal{A}[\epsilon]^n)}(X,Y)\cong \Hom_{\D(\mathcal{A}[\epsilon]^n)}(X,Y).$$
\end{proposition}

\begin{proof}
Consider the canonical map
$$G: \Hom_{\K(\mathcal{A}[\epsilon]^n)}(X,Y)\to \Hom_{\D(\mathcal{A}[\epsilon]^n)}(X,Y)$$
defined by $G(f)=f/1_X$. If $G(f)=f/1_X=0$, then by Lemma ~\ref{lem: 5.8}(1), there exists a roof
$$X \buildrel {s} \over \Leftarrow X' \buildrel {0} \over \longrightarrow Y$$
such that $s$ is a quasi-isomorphism, which is equivalent to the roof
$$X \buildrel {1_X} \over \Leftarrow X \buildrel {f} \over \longrightarrow Y.$$
Hence we have $fs\thicksim 0$. It follows from Lemma ~\ref{lem: 5.15} that there exists a morphism
$g: X\to X'$ such that $sg\thicksim 1_X$. Thus $f\thicksim 0$. On the other hand, let $f/s\in \Hom_{\D(\mathcal{A}[\epsilon]^n)}(X,Y)$,
that is, it has the form
$$X \buildrel {s} \over \Leftarrow Z \buildrel {f} \over \longrightarrow Y.$$
By Lemma ~\ref{lem: 5.15} again, there exists a morphism $g: X\to Z$ such that $sg\thicksim 1_X$.
Then we obtain that $f/s=fg/1_X=G(fg)$ and $G$ is an isomorphism.
\end{proof}

We end this section with the following result.

\begin{theorem}\label{thm: 5.17}
Assume that $\mathcal{A}$ is an Ab4-category with a compact projective generator.
Then $\D(\mathcal{A}[\epsilon]^n)$ is a compactly generated triangulated category.
\end{theorem}

\begin{proof}
Let $G$ be a compact projective generator in $\mathcal{A}$. Firstly, $\D(\mathcal{A}[\epsilon]^n)$ has arbitrary coproducts
by Corollary ~\ref{cor: 5.13}. For any $1\leqslant i\leqslant n-1$, we use $T^i(G)$ to denote the $n$-th differential module
$(G^i, \epsilon_{G^i})$, where
$$\epsilon_{G^i}:= \left(
\begin{array}{ccccc}
0& 0& 0&\cdots &0 \\
1& 0& 0&\cdots &0 \\
0& 1& 0&\cdots &0 \\
\vdots&\vdots &\ddots& \ddots&\vdots\\
0& 0& \cdots&1 &0\\
\end{array}
\right)_{i\times i.}$$

{\bf Claim.} For any $X\in \K(\mathcal{A}[\epsilon]^n)$ and $1\leqslant i\leqslant n-1$, we have
$$\Hom_{\K(\mathcal{A}[\epsilon]^n)}(T^i(G), X)\cong \H_{(i)}(\Hom_{\mathcal{A}}(G, X)).$$

Given $f=(f_1,f_2,\cdots, f_i)\in \Hom_{\mathcal{A}[\epsilon]^n}(T^i(G), X)$, we define
$$\theta: \Hom_{\mathcal{A}[\epsilon]^n}(T^i(G), X)\to \H_{(i)}(\Hom_{\mathcal{A}}(G, X))$$
via $\theta(f)=f_1+\im {\Hom_{\mathcal{A}}(G,\epsilon^{n-i}_X})$. Since the equality $\epsilon_Xf=f\epsilon_{G^i}$ holds,
we immediately get $\epsilon_Xf_i=0$ and $\epsilon_Xf_j=f_{j+1}$ for any $1\leqslant j\leqslant i-1$. Thus
$${\epsilon^i_X}f_1={\epsilon^{i-1}_X}f_2=\cdots = \epsilon_Xf_i=0.$$
It means that $\theta$ is well defined. Let $f+\im {\Hom_{\mathcal{A}}(G,\epsilon^{n-i}_X})\in \H_{(i)}(\Hom_{\mathcal{A}}(G, X))$.
Then ${\epsilon^i_X}f=0$. Set $f_1:=f$ and $f_j:=\epsilon_X^{j-1}f_1$ for any $2\leqslant j\leqslant i$. Then
$\theta(f_1,f_2,\cdots,f_i)=f+\im {\Hom_{\mathcal{A}}(G,\epsilon^{n-i}_X})$, which implies that $\theta$ is surjective.
If $\theta(f)=f_1+\im {\Hom_{\mathcal{A}}(G,\epsilon^{n-i}_X})=0$, then there exists $h\in \Hom_{\mathcal{A}}(G,X)$
such that $\epsilon_X^{n-i}h=f_1$. Set $g_i:=h$ and $s:=(0,0,\cdots,g_i): T^i(G)\to X$. It is easily seen that
$$f=\epsilon_{X}^{n-1}s+\epsilon_{X}^{n-2}s\epsilon_{G^i}+\cdots + s\epsilon_{G^i}^{n-1}$$
and $f$ is null-homotopic. The claim is proved.

By the above claim, we know that $T^i(G)$ is $K$-projective for any $1\leqslant i\leqslant n-1$ since $G$ is projective.
If $X\in \D(\mathcal{A}[\epsilon]^n)$ such that $\Hom_{\D(\mathcal{A}[\epsilon]^n)}(T^i(G), X)=0$
for any $1\leqslant i\leqslant n-1$. Then it is deduced from Proposition ~\ref{prop: 5.16} that
$$\Hom_{\D(\mathcal{A}[\epsilon]^n)}(T^i(G), X)\cong \Hom_{\K(\mathcal{A}[\epsilon]^n)}(T^i(G), X)\cong \H_{(i)}(\Hom_{\mathcal{A}}(G, X))=0.$$
Since $G$ is a generator, it implies $X=0$ in $\D(\mathcal{A}[\epsilon]^n)$ by \cite[Lemma 3.1]{G1}.
Let $\{X_j\}_{j\in J}$ be a family of objects of $\D(\mathcal{A}[\epsilon]^n)$.
Using Proposition ~\ref{prop: 5.16} and the above claim again, we have
\begin{align*}
&\Hom_{\D(\mathcal{A}[\epsilon]^n)}(T^i(G), \coprod _{j\in J}X_j)\cong  \Hom_{\K(\mathcal{A}[\epsilon]^n)}(T^i(G), \coprod _{j\in J}X_j)\\
& \ \ \ \ \ \ \ \ \ \ \ \ \ \ \ \ \ \ \ \ \ \ \ \ \ \ \ \ \ \ \ \ \ \ \ \ \cong \H_{(i)}((\Hom_{\mathcal{A}}(G,\coprod _{j\in J}X_j))\\
& \ \ \ \ \ \ \ \ \ \ \ \ \ \ \ \ \ \ \ \ \ \ \ \ \ \ \ \ \ \ \ \ \ \ \ \ \cong \coprod _{j\in J}\H_{(i)}((\Hom_{\mathcal{A}}(G, X_j)
\ (\text{since $G$ is compact})\\
&\ \ \ \ \ \ \ \ \ \ \ \ \ \ \ \ \ \ \ \ \ \ \ \ \ \ \ \ \ \ \ \ \ \ \ \ \cong \coprod _{j\in J}\Hom_{\K(\mathcal{A}[\epsilon]^n)}(T^i(G),X_j)\\
& \ \ \ \ \ \ \ \ \ \ \ \ \ \ \ \ \ \ \ \ \ \ \ \ \ \ \ \ \ \ \ \ \ \ \ \ \cong \coprod _{j\in J}\Hom_{\D(\mathcal{A}[\epsilon]^n)}(T^i(G),X_j).
\end{align*}
So $T^i(R)$ is a compact object in $\D(\mathcal{A}[\epsilon]^n)$, proving the assertion.
\end{proof}

As an immediate consequence of Theorem \ref{thm: 5.17}, we get the following

\begin{corollary}\label{cor: 5.18}
$\D((\Mod R)[\epsilon]^n)$ is a compactly generated triangulated category.
\end{corollary}

\vspace{0.5cm}

{\bf Acknowledgements.}
This paper was initialed during the first named author's visit to Northeastern University
from May 2018 to May 2019; he thanks Alex Martsinkovsky for his hospitality. This research
was partially supported by NSFC (Grant Nos. 11971225, 11571164, 11501144),
a Project Funded by the Priority Academic Program Development of Jiangsu Higher Education Institutions
and NSF of Guangxi Province of China (Grant No. 2016GXNSFAA380151).
The authors thank the referee for the useful comments.

\providecommand{\bysame}{\leavevmode\hbox to3em{\hrulefill}\thinspace}
\providecommand{\MR}{\relax\ifhmode\unskip\space\fi MR }
\providecommand{\MRhref}[2]{%
\href{http://www.ams.org/mathscinet-getitem?mr=#1}{#2}
}
\providecommand{\href}[2]{#2}

\end{document}